\documentclass[12pt]{article}
\usepackage{authblk}
\usepackage{comment}
\usepackage{color}
\usepackage{amssymb}
\usepackage{amsthm}
\usepackage{amscd}
\usepackage{graphicx}
\usepackage{epstopdf}
\usepackage{multicol}
\usepackage{multirow}
\usepackage{subfigure}
\usepackage{titlesec}
\titleformat{\section}{\bfseries}{\thesection .}{0.5em}{}
\titleformat{\subsection}{\itshape}{\thesubsection .}{0.5em}{}
\usepackage[top=25.4mm, bottom=25.4mm, left=31.7mm, right=32.2mm]{geometry}
\usepackage[tbtags]{amsmath}
\usepackage{booktabs}
\allowdisplaybreaks
\usepackage[numbers,sort&compress]{natbib}
\usepackage{multirow}
\usepackage{dsfont}
\usepackage{algorithm}
\usepackage{algorithmic}

\usepackage{caption}
\usepackage{color}
\usepackage{bm}
\usepackage{amsmath}
\usepackage{amsfonts}
\usepackage{graphicx,amsmath,amsfonts,amssymb,amscd,amsthm}
\usepackage{enumerate}
\usepackage{amsthm,amsmath,amssymb}
\usepackage{mathrsfs}
\newtheoremstyle{kai}
{3pt} {3pt} {} {} {\bfseries} {.} {.5em} {}
\def\EquationsBySection{\def\theequation
{\thesection.\arabic{equation}}%
\@addtoreset{equation}{section}}

\newcommand\old[1]{}
\newcommand{\pend}{\hfill \thicklines \framebox(6.6,6.6)[l]{}}
\renewenvironment{proof}{\noindent {\it  Proof.} \rm}{\pend}
\newtheorem{theorem}{Theorem}
\newtheorem{lemma}{Lemma}

\newtheorem{proposition}{Proposition}

\newtheorem{remark}{Remark}
\newtheorem{definition}{Definition}
\newtheorem{example}{Example}
\usepackage{indentfirst}
\usepackage[pagewise]{lineno}

\begin{document}
\pagestyle{plain}
\title
{ Nonstationary nonzero-sum Markov games
\\under a  probability criterion}
\author{Xin Guo$^1$, and Xin Wen$^2$\thanks{The email address: guox87@mail.sysu.edu.cn (X. Guo), wenx33@mail3.sysu.edu.cn (X. Wen). This
work has been supported by the National Key Research and Development Program of China (2022YFA1004600) and  the National
Natural Science Foundation of China (12301170, 72301304, 72342006, 11931018).} \\
\vspace{0.8cm}
\small \it
    {$^1$School of Science, Sun Yat-sen
    	University, Guangzhou 510275, China \\
    \it	$^2$School of Business, Sun Yat-sen
    	University, Guangzhou 510275, China}
}

\date{}
\maketitle

\begin{abstract} This  paper deals with $N$-person nonzero-sum discrete-time Markov games under a probability criterion, in which the transition probabilities and reward functions are allowed to vary with time.
Differing from the existing works on the expected reward criteria, our concern here is to maximize the probabilities that the accumulated rewards until the first passage time to any target set exceed a given goal, which represent the reliability of the players' income. Under a mild condition, by developing a comparison theorem for the  probability criterion, we prove the existence of a Nash equilibrium over history-dependent policies. Moreover, we provide an efficient  algorithm for computing  $\epsilon$-Nash equilibria. Finally, we illustrate our main results by a nonstationary energy management model and take a numerical experiment.

\end{abstract}


\section{Introduction}\label{sec1}

{Stochastic games were first introduced in Shapley's seminal work \cite{Shapley1953},  whose historic context and importance of this work  are well documented in \cite{Solan2015}.
Following Shapley’s foundational contribution \cite{Shapley1953}, zero-sum and nonzero-sum stochastic games
receive particular attention in subsequent research. In zero-sum game theory, Jaśkiewicz and Nowak \cite{Jaskiewicz2017-1} present a nice review on the basic streams of research in this area, in the
context of the existence of value and uniform value, algorithms, vector payoffs,
incomplete information, and imperfect state observation.
For nonzero-sum game, the same authors \cite{Jaskiewicz2017-2} also provide a comprehensive survey encompassing key research directions such as  the existence of
stationary Nash equilibria, correlated equilibria, subgame-perfect equilibria and uniform equilibria, and stopping games.
Furthermore, both papers \cite{Jaskiewicz2017-1,Jaskiewicz2017-2} introduce the expected discount criterion and the limit-average criterion which are common optimality criteria in stochastic games. Below, we will introduce the corresponding developments of three common criteria in zero-sum and non-zero-sum stochastic games:} the (expected) discounted criterion \cite{2Guo2005,Flesch2018,Babu2017,Neyman2017,Pun2016}, the average criterion \cite{Jaskiewicz2002,Babu2017,Simon2016} and the risk-sensitive criterion
 \cite{Bauerle2017,Golui2023,Ghosh2023,Basu2018}.

More precisely,  for the expected discounted criterion, Flesch et al. \cite{Flesch2018} study the  zero-sum discrete-time Markov games (DTMGs) and give the existence conditions of the stationary Nash equilibrium.
Moreover, Babu et al. \cite{Babu2017} consider zero-sum and nonzero-sum DTMGs, and proved the existence of the corresponding Nash equilibrium policies. For continuous-time Markov games (CTMGs), Guo and {Hern\'{a}ndz-Lerma} discuss the expected discounted criterion  \cite{2Guo2005} with a possible unbounded transition rate function for zero-sum games, and the existence of the optimal stationary policy pair is guaranteed.
Neyman \cite{Neyman2017} studies nonzero-sum CTMGs and proves the existence of a Nash equilibrium under the condition that the transition rate and reward function are both bounded. Pun and Wong \cite{Pun2016} study the stochastic nonzero-sum differential insurance game under a general expected total reward criterion, and obtain the existence of a Nash equilibrium.

For the expected average criterion,  Babu et al. \cite{Babu2017} consider the existence of the corresponding Nash equilibrium policies for the zero-sum and nonzero-sum DTMGs. Simon \cite{Simon2016} discusses the existence of approximate Nash equilibrium for nonzero-sum DTMGs.   Ja\'{s}kiewicz \cite{Jaskiewicz2002} studies two-person zero-sum semi-Markov games (SMGs) under the average criterion, where it is assumed that the transition probability satisfies some generalized geometric ergodic conditions, and proves that the solution of Shapley equation can be approximated by the solution of some Shapley equations for $\epsilon$-perturbed SMGs.

Moreover, for risk-sensitive criterion,  B\"{a}uerle and Rieder \cite{Bauerle2017} consider the   zero-sum DTMGs and prove the
existence of optimal policies.
Very recently,
Golui and Pal \cite{Golui2023} study the  zero-sum CTMGs on a general state space with risk-sensitive discounted cost criteria. { Ghosh} et al. \cite{Ghosh2023} consider the  zero-sum DTMGs for risk-sensitive average cost criterion   with countable/compact state space and Borel action spaces. Guo et al. \cite{GuoX2023}  discuss the expected discounted zero-sum risk-sensitive stochastic games with unbounded payoff functions and varying discount factors. 
 For nonzero-sum stochastic games, Basu and {Ghosh} \cite{Basu2018} study  the infinite horizon risk-sensitive discounted cost and ergodic cost  for controlled Markov chains with countably many states. 
 
 
As we can see, these studies \cite{Basu2018,Pun2016,Simon2016,Neyman2017,Ghosh2023,Golui2023,Bauerle2017,Babu2017,Flesch2018,2Guo2005,Bayraktar2016,De2018,Jaskiewicz2002} focus on the (expected) discounted, the average and the risk-sensitive criteria.
However, probability, as an effective and intuitive way to measure risk, is often chosen as an important risk metric.
Formally, the security or safety of player's rewards can be defined as the probability that the player's accumulated payoff reaches to some given profit goal before the state drops into a target set. 
 Recently, Bhabak and Saha \cite{Bhabak2021} consider CTMGs for the  probability criterion,  and show the existence of the value of the game and a randomized stationary saddle point.
In \cite{Huang2023}, a two-person zero-sum SMGs  is discussed, and the existence condition of the optimal policy pair is proved and an approximate algorithm is proposed.
 For the nonzero-sum stochastic game with probability criterion, as far as we know, only  \cite{Huang2020} discusses and proves the existence of Nash equilibrium.
As a special case of the game problem, Markov decision processes (MDPs in short) also optimize the players' decisions, but there is only one decision maker. 
Some achievements have been made in the research of MDPs with probability criteria; see \cite{Wen} for discrete-time case, \cite{21} for semi-Markov case and \cite{11} for continuous-time case.

All of these papers \cite{Basu2018,Pun2016,Simon2016,Neyman2017,Ghosh2023,Golui2023,Bauerle2017,Babu2017,Flesch2018,2Guo2005,Bayraktar2016,De2018,Jaskiewicz2002,Huang2023,11,Wen,21,Bhabak2021}
mentioned above assumed that the payoff functions and the transition laws is time-homogeneous.
However, in many fields like
actual economies, experience population growth and technological progress, the payoff functions and the transition probabilities usually change with time.
Thus, it is natural and desirable  to discuss the  nonstationary stochastic games. In fact,  the nonstationary zero-sum games have been studied in \cite{Nowak4, Leao} {\color{black}and the nonstationary nonzero-sum games have been studied in \cite{Maitra, Mertens}}. Precisely, Nowak \cite{Nowak4} considers the expected total payoff criterion  of the nonstationary zero-sum games and  proves the existence of the game's  value and optimal policies.
Leao et al. \cite{Leao} consider that
the players do not have all information about the states of the
game based on the model in \cite{Nowak4},
and establish the value of the game and optimal policies.  {\color{black}Maitra and Sudderth consider $n$-person nonzero-sum games with the time-dependent payoff function in \cite{Maitra}.}

To the best of our knowledge,  the nonzero-sum   nonstationary stochastic games {\color{black}with the probability criterion}  have not been studied.  In addition, the above literature on nonzero-sum stochastic games under probability criterion mainly discusses the existence conditions of a Nash equilibrium, and there is {\color{black}no} literature on the effective algorithm, which motivates the study in this paper.

Aforementioned problems intrigue our interests in the more general case: the nonstationary $N$-person nonzero-sum games with probability criterion. We make efforts in this paper to maximize the probabilities of the accumulated rewards until the first passage time to any target set exceeding a given goal, which represents the reliability of the players' income. Different from the case of stationary DTMGs under the probability criterion with only one optimality equation, we derive a sequence of optimality equations for the nonstationary  model, and give a new condition (i.e., Proposition 1 below), which is the generalization of those in \cite{Huang2020}. Under the new condition, we present a comparison theorem for the probability criterion, and establish the existence of a solution to the optimality equations. Moreover, using the comparison theorem and optimality equations developed here, we also prove the {\color{black}existence} of a Nash equilibrium. Moreover, we provide an efficient
algorithm to find $\epsilon$-Nash equilibria, and give the convergence and error bound analyses of the algorithm.
Furthermore, we use an   electricity system to demonstrate our main results in which our new condition is satisfied but the previous ones in \cite{Huang2020} for stationary cases fail to hold.

The remainder of this paper structures as follows. Section 2 builds the model of nonstationary $N$-person nonzero-sum games with probability criterion. In Section 3, we give a new and {\color{black}mild} condition, under which the comparison theorem is established. Section 4 is the main proof of the existence of the Nash equilibrium, which is based on our comparison theorem. An efficient algorithm is developed in Section 5. Finally we illustrate the results with an electricity system example in Section 6.

\section{The game  model}\label{sec2}
Let $I=\{1,2,\ldots,N\}$, $\mathbb{N}_0:=\{0,1,\cdots\}$. The model of a nonstationary $N$-person nonzero-sum stochastic dynamic game is of the following form
\begin{eqnarray}\label{model}
\{S_n,D, A_n^k,A_n^k(i),p_n(j|i,\bm{a}),r_n^k(i,\bm{a})\}_{n\in \mathbb{N}_0},
\end{eqnarray}
where
\\$\bullet$ $S_n$:  denotes the state space at time $n$, assumed to be  denumerable;
\\$\bullet$ $D$:  is a given target set, which is assumed to be nonempty subset of $\cap_{n=0}^{\infty}S_n$;
\\$\bullet$ $A_n^k$:  denotes the action space for player $k$ at time $n$;
\\$\bullet$ $A_n^k(i)$: is a finite set of  actions for player {\color{black}$k$},  when the system is in state $i\in S_n$ at time $n$;
\\$\bullet$ $p_n(j|i,\bm{a})$: represents the transition probability of states from $i\in S_n$ to $j\in S_{n+1}$ under action $\bm{a}\in \bm{A}_n(i)$,where $\bm{A}_n(i):=\times_{k=1}^{N}A_n^k(i)$;
\\$\bullet$ $r_n^k(i,\bm{a})$: a nonnegative rational-valued  function on $\mathbb{K}_n$, {\color{black}denoting} the reward at stage $n$ for player $k$,
where $\mathbb{K}_n:=\{(i,\bm{a})|i\in S_n, \text{a}\in \bm{A}_n(i)\}$ ($n\in\mathbb{N}_0$) is   a subset of $S_n\times \bm{A}_n$ with  $\bm{A}_n:=\times_{k=1}^{N}A_n^k$.

We now describe the evolution of a nonstationary $N$-person nonzero-sum game with the probability criterion. Suppose that a system occupies state $i_0\in S_0$ at initial decision time 0. Given any  (rational-valued)  profit  goal $\bm{\lambda}_0:=(\lambda_0^1,\cdots,\lambda_0^N)$ for players, that is, each player $k$ should try the best to obtain rewards more than $\lambda_0^k$  before the system state falls into the  target set $D$.  Based on the system state $i_0$ and the goal $\bm{\lambda}_0$, the players  independently choose actions $\bm{a}_0:=(a_0^1,\cdots,a_0^N)$ from $\bm{A}_0(i_0)$.
As the consequences of the action choices, the following events occur: (i) at time 1 the system state changes to $i_1\in S_1$ according to the transition probability $p_0(i_1\mid i_0,\bm{a}_0)$; (ii) player $k$ receives a reward $r_0^k(i_0,\bm{a}_0)$, and then the remaining goal   is $\lambda^k_1:=\lambda^k_0-r_0^k(i_0,\bm{a}_0)$ for each $k\in I$. According to the current state $i_1$ and the current goal $\bm{\lambda}_1:=(\lambda_1^1,\ldots,\lambda_1^N)$ as well as the previous state $i_0$ and goal $\bm{\lambda}_0$, the  action $\bm{a}_1\in \bm{A}_1(i_1)$ is chosen, and the same sequence of events occurs. The process evolves in this way. Up to  time $n$, we  obtain an admissible history $(i_0,\bm{\lambda}_0,\bm{a}_0,\ldots,i_{n-1},\bm{\lambda}_{n-1},\bm{a}_{n-1} ,i_n,\bm{\lambda}_n)$,
where $\bm{\lambda}_m:=(\lambda_m^1,\ldots,\lambda_m^N)$ , $\lambda^k_{m+1}:=\lambda^k_m-r^k_m(i_m,\bm{a}_m)$, $(i_m,\bm{a}_m)\in\mathbb{K}_m$, $m=0,1,\ldots,n-1$ and $i_n\in S_n$, $n\in\mathbb{N}_0$.

For the following arguments, let $\mathbb{Q}$  be the set of all rational numbers, $\mathbb{H}_0:={\color{black}S_0\times \mathbb{Q}^N}$,
$\mathbb{H}_n:=\{(h_{n-1},\bm{a}_{n-1},i_n,\bm{\lambda}_n),h_{n-1}\in \mathbb{H}_{n-1},\bm{a}_{n-1}\in \bm{A}_{n-1}(i_{n-1}),  i_n\in S_n,\bm{\lambda}_n\in\mathbb{Q}^N\}$, $n\ge 1$, which are endowed with the Borel $\sigma$-algebra.

To precisely define the optimality criterion, we need to introduce the concepts of policies.
\begin{definition}
A randomized history-dependent policy of player $k$ is a sequence $\pi^k= \{\pi_n^k,n\ge0\}$ of stochastic kernels $\pi_n^k$ on $A_n^k$ given $\mathbb{H}_n$, which satisfies the constraint $\pi_n^k(A_n^k(i_n)\mid h_n)=1$ for every $h_n=(i_0,\bm{\lambda}_0,\bm{a}_0,\ldots,i_{n-1},\bm{\lambda}_{n-1},\bm{a}_{n-1},i_n,\bm{\lambda}_n)\in \mathbb{H}_n$ and $n\in\mathbb{N}_0$. The set of all randomized history-dependent  policies for player $k$ is denoted by $\Pi^k$.

Let $\Phi_n^k$ ($n\in\mathbb{N}_0,k\in I$) denote the set of all stochastic kernels $\phi_n^k$ on $A_n^k$ given $S_n\times \mathbb{Q}^N$ such that $\phi_n^k(A_n^k(i)\mid i,\bm{\lambda})=1$ for all $ (i,\bm{\lambda})\in S_n\times \mathbb{Q}^N$.

A policy $\pi^k= \{\pi_n^k,n\in \mathbb{N}_0\}\in\Pi^k$ is said to be randomized Markov for player $k(k\in I)$ if there exists a sequence $\{\phi_n^k,n\ge0\}$ of stochastic kernels $\phi_n^k\in\Phi_n^k$ satisfying $\pi_n^k(\cdot\mid h_n)=\phi_n^k(\cdot\mid i_n,\bm{\lambda}_n)$ for every $h_n=(i_0,\bm{\lambda}_0,\bm{a}_0,\ldots,i_{n-1},\bm{\lambda}_{n-1},\bm{a}_{n-1} ,i_n,\bm{\lambda}_n)\in \mathbb{H}_n$ and $n\in\mathbb{N}_0$.
We write such a policy as $\pi^k=\{\phi_n^k,n\ge0\}$.
The set of all randomized Markov policies for player $k$ is denoted by $\Pi^k_m$.
Then $\bm{\Pi}:=\times_{k=1}^N\Pi^k$ and $\bm{\Pi}_m:=\times_{k=1}^N\Pi^k_m$ are the corresponding multipolicies.
\end{definition}

Given any  $\bm{\pi}=(\pi^1,\pi^2,\ldots,\pi^N)\in\bm{\Pi}$, $k\in I$, and policy $\pi'\in\Pi^k$, we take $\bm{\pi}^{-k}:=(\pi^l,l\in I,l\ne k)$ as the $N-1$ dimensional multipolicy and $[\bm{\pi}^{-k},\pi']$ denote the multipolicy that player $l$ uses $\pi^l$ for each $l\ne k$ and player $k$ uses $\pi'$. Similarly, $\bm{\Pi}^{-k}$ and $\bm{\Pi}_m^{-k}$ are the corresponding multipolicies.
Furthermore, for any fixed $n\in \mathbb{N}_0$, we denote
$\bm{\pi}_n:=(\pi^1_n,\pi^2_n,\ldots,\pi^N_n)$ and  take $\bm{\pi}_n^{-k}:=(\pi_n^l,l\in I,l\ne k)$ as the $N-1$ dimensional policy.

For each initial distribution $\gamma$ on $S_0\times \mathbb{Q}$ and $\bm{\pi}\in\bm{\Pi}$, by the well-known {\color{black}Ionescu-Tulcea theorem} \cite{Hernandez1996}, there exists a unique probability space $(\Omega,\mathbb{P}^{\bm{\pi}}_{\gamma})$ and a stochastic process $\{X_n,\Lambda_n^1, $\ldots,$\Lambda_n^N,Y_n^1, $\ldots,$ Y_n^N,n\in \mathbb{N}_0\}$, where $X_n$, $\Lambda_n^k$, $Y_n^k$ represent the random variables of the state, the profit goal, and the action adopted by player $k$ at time $n$, respectively. Let $$\bm{\Lambda}_n=(\Lambda_n^1,\ldots,\Lambda_n^N), Y_n=(Y_n^1,\ldots,Y_n^N),
  H_n=(X_0,\bm{\Lambda}_0,Y_0,
  \ldots,X_{n-1},\bm{\bm{\Lambda}}_{n-1},Y_{n-1},\ldots,X_n,\bm{\Lambda}_n).$$ Then, for any subset $C\subseteq\mathbb{Q}$,  $h_n=(i_0,\bm{\lambda}_0,\bm{a}_0,\ldots,i_{n-1},\bm{\lambda}_{n-1},\bm{a}_{n-1} , i_n,\bm{\lambda}_n)\in \mathbb{H}_n$,  $i_{n+1} \in S_{n+1}$, $\bm{a}_n=(a_n^1,\ldots,a_n^N)\in \bm{A}_n(i_n)$ and $n\ge 0$, we have
\begin{eqnarray}
&&\mathbb{P}_{\gamma}^{\bm{\pi}}(Y_n=\bm{a}_n\mid H_n=h_n)
=\pi_n^1(a_n^1\mid h_n)\pi_n^2(a^2_n\mid h_n)\ldots \pi_n^N(a^N_n\mid h_n),\label{measure}\\
&&\mathbb{P}_{\gamma}^{\bm{\pi}}(\Lambda_{n+1}^k\in C\mid H_n=h_n,Y_n=\bm{a}_n)=\mathbb{I}_C(\lambda^k_{n}-r_{n}^k(i_{n},\bm{a}_{n}))~\forall~k\in I,\label{lambda}\\
&&\mathbb{P}_{\gamma}^{\bm{\pi}}(X_{n+1}=i_{n+1}\mid H_n=h_n,Y_n=\bm{a}_n)
=p_n(i_{n+1}\mid i_n,\bm{a}_{n}),\label{P}
          \end{eqnarray}
where the $\mathbb{I}_C$ denotes the indicator function of a set $C$.
The expectation operator with respect to $\mathbb{P}^{\bm{\pi}}_{\gamma}$ is denoted by $\mathbb{E}^{\bm{\pi}}_{\gamma}$.

In order to introduce the     probability criterion, let 
\begin{eqnarray*}
\tau_D:=\inf\{n\ge 0,X_n\in D\}~~(with~ \inf \emptyset:=\infty),
\end{eqnarray*}
which represents the first passage time into the target set $D$ of  the state process.

For any  $i\in S_0$, $\bm{\lambda}=(\lambda^1,\ldots,\lambda^N)\in \mathbb{Q}^N$ and $\bm{\pi}\in\bm{\Pi}$, the  probability criterion for player $k~(k\in I)$ is defined by
\begin{eqnarray}\label{tau}
F^k(i,\bm{\lambda},\bm{\pi}):=\mathbb{P}_{(i,\bm{\lambda})}^{\bm{\pi}}
(\sum_{n=0}^{\tau_D-1}r_n^k(X_n,Y_n)\geq \lambda^k).
\end{eqnarray}
where $\sum_{n=l}^{m}y_n:=0$   for any sequence $\{y_n\}$ when $m<l$.

$F^k(i,\bm{\lambda},\bm{\pi})$ measures the security  for player $k$ that the total rewards obtained is more than the profit goal $\lambda^k$ under the policy $\bm{\pi}\in\bm{\Pi}$. Moreover,
for all $k\in I$, $(i,\bm{\lambda})\in D\times \mathbb{Q}^N$ and $\bm{\pi}\in\bm{\Pi}$, since $\tau_D=0$, so $\sum_{n=0}^{\tau_D-1}r_n^k(X_n,Y_n)=0$. Thus, we have $F^k(i,\bm{\lambda},\bm{\pi})=\mathbb{P}_{(i,\bm{\lambda})}^{\bm{\pi}}
(\sum_{n=0}^{\tau_D-1}r_n^k(X_n,Y_n)\geq \lambda^k)=\mathbb{P}_{(i,\bm{\lambda})}^{\bm{\pi}}
(0\geq \lambda^k)=
\mathbb{I}_{(-\infty,0]}(\lambda^k)$.

   Define the optimal $\bm{\pi}$-response function for player $k$ by
\begin{eqnarray}\label{n-value}
F^{k}_*(i,\bm{\lambda},\bm{\pi}^{-k}):=\sup_{\pi^k\in\Pi^k}
F^k(i,\bm{\lambda},\bm{\pi})~\forall~(i,\bm{\lambda})\in S_0\times \mathbb{Q}^N.
\end{eqnarray}
\begin{definition}
A policy $\bm{\pi}\in \bm{\Pi}$ is called a Nash equilibrium for the nonzero-sum game if
\begin{eqnarray}
F^k(i,\bm{\lambda},\bm{\pi})=F^{k}_*(i,\bm{\lambda},\bm{\pi}^{-k})~~\forall ~(i,\bm{\lambda})\in S_0\times \mathbb{Q}^N, k\in I.
\end{eqnarray}
\end{definition}

Our goal here is to give conditions for the existence and computation of  Nash equilibria.

\section{Preliminaries}\label{sec3}

This section gives some facts, which are required for the arguments of our main results about the existences of  Nash equilibria.

To achieve our goal, besides the first passage time $\tau_D$, we introduce another first passage time $\tau_D^n$ into the target set $D$ after time $n$ for any $n\geq 0$:
\begin{eqnarray*}
	\tau_D^n:=\inf\{k\ge n,X_k\in D\}~\forall~ n\ge 0.
\end{eqnarray*}

Using  $\tau_D^n$ and the Markov property of  $\{X_n,\bm{\Lambda}_n,Y_n,n\geq 0\}$ under any Markov policy,  we can reasonably introduce the following notation: for  $n\ge 1$, $i\in S_n$, $\bm{\lambda}=(\lambda^1,\ldots,\lambda^N)\in \mathbb{Q}^N$ and $\bm{\pi}\in \bm{\Pi}_m$,
\begin{eqnarray*}
F_n^{k}(i,\bm{\lambda},\bm{\pi}):=\mathbb{P}_{\gamma}^{\bm{\pi}}
(\sum_{t=n}^{\tau_D^n-1}r_t^k(X_t,Y_t)\geq \lambda^k\mid X_n=i,\bm{\Lambda}_n=\bm{\lambda}),
\end{eqnarray*}
 which represents the risk probability of player $k$ from any time $n$ to $\tau_D^n$ under the policy $\bm{\pi}\in \bm{\Pi}_m$ when the state and goal at time $n$ are $i$ and $\bm{\lambda}$, respectively. Moreover, due to  $\tau_D=\tau_D^0$, for consistency, we denote $F_0^k(i,\bm{\lambda},\bm{\pi}):=F^k(i,\bm{\lambda},\bm{\pi})$ for all ~$(i,\bm{\lambda})\in S_0\times \mathbb{Q}^N, k\in I$ and $\bm{\pi}\in\bm{\Pi}$.

To further establish the comparison theorem for the probability criterion and prove the uniqueness of the solution to the optimality equations, we impose {\color{black}the} assumption below.

{\bf Assumption A.}
For each $\bm{\pi}\in \bm{\Pi}$ and $n\ge 0$, $\mathbb{P}_{\gamma}^{\bm{\pi}}
(\tau_D^n<\infty\mid H_n=h_n)=1$ for $h_n=(i_0,\bm{\lambda}_0,\bm{a}_0,\ldots ,i_n,\bm{\lambda}_n)\in\mathbb{H}_n$ with $i_n\in S_n\setminus D$ and $\mathbb{P}_{\gamma}^{\bm{\pi}}(H_n=h_n)>0$, where  $S_n\setminus D$ represents the complement of $D$ with respect to $S_n$.

Assumption A implies that starting at any fixed time $n\ge 0$, the controlled system will eventually arrive at $D$ within finite time under any policy $\bm{\pi}\in \bm{\Pi}$ when the state at time $n$ belongs to $S_n\setminus D$.

To verify Assumption A, we provide a  sufficient condition, which is imposed on the primitive {data} of model (\ref{model}) and is easy to be verified.
\begin{proposition}\label{prop1}
For each $n\ge 0$,  if there exists a  constant $\beta_n\in[0,1)$ such that $p_n(D|i,\bm{a})\ge \beta_n$ for all $i\in S_n\setminus D$, $\bm{a}\in \bm{A}_n(i)$ and  $\sum_{n=0}^{\infty}\beta_n=+\infty$, then Assumption A holds.
\end{proposition}
\begin{proof} For any given $\bm{\pi}\in \bm{\Pi}$ and $n\ge 0$, we now  prove by induction that
\begin{eqnarray}\label{induction1}
\mathbb{P}_{\gamma}^{\bm{\pi}}
(\tau_D^n>n+k\mid H_n=h_n)&=&\mathbb{P}_{\gamma}^{\bm{\pi}}
(\bigcap_{m=n+1}^{k+n}\{X_m\in S_m\setminus D\}\mid H_n=h_n)\nonumber\\
&\leq&(1-\beta_n)(1-\beta_{n+1})\cdots(1-\beta_{n+k-1})~~\forall~k\ge 1,
\end{eqnarray}
when $\mathbb{P}_{\gamma}^{\bm{\pi}}(H_n=h_n)>0$ for $h_n=(i_0,\bm{\lambda}_0,\bm{a}_0,\cdots ,i_n,\bm{\lambda}_n)\in \mathbb{H}_n$, $i_n\in S_n\setminus D$.
\\For any $n\ge 0$ and $k=1$,  since
$p_n(D|i,a) \ge \beta_n$, we have
\begin{eqnarray*}
&&\mathbb{P}_{\gamma}^{\bm{\pi}}
(\tau_D^n>n+1\mid H_n=h_n)\\
&=&\mathbb{P}_{\gamma}^{\bm{\pi}}
(X_{n+1}\in S_{n+1}\setminus D\mid H_n=h_n)\\
&=&
\mathbb{E}_{\gamma}^{\bm{\pi}}
\left[\mathbb{I}_{\{X_{n+1}\in S_{n+1}\setminus D\}}\mid H_n=h_n\right]\\
&=&\sum_{a_n^1\in A_n^1(i_n)}\cdots\sum_{a_n^N\in A^N_n(i_n)}\sum_{i_{n+1}\in S_{n+1}\setminus D}p_n(i_{n+1}\mid i_n,\bm{a})\pi_n^1(a_n^1\mid h_n)\cdots \pi_n^N(a^N_n\mid h_n)\\
&=&\sum_{a_n^1\in A_n^1(i_n)}\cdots\sum_{a_n^N\in A^N_n(i_n)}(1-p_n(D\mid i_n,\bm{a}_n))\pi_n^1(a_n^1\mid h_n)\cdots \pi_n^N(a^N_n\mid h_n)\\
&\leq& 1-\beta_n.
\end{eqnarray*}
 Now suppose that (\ref{induction1}) holds for some $k\ge 1$ and any $n\ge0$. Then for $k+1$, we have
\begin{eqnarray*}
&&\mathbb{P}_{\gamma}^{\bm{\pi}}
(\tau_D^n>n+k+1\mid H_n=h_n)\\
&=&
\mathbb{P}_{\gamma}^{\bm{\pi}}\left(\bigcap_{m=n+1}^{k+n+1}\{X_m\in S_m\setminus D\}\Big{|} H_n=h_n\right)\\
&=&\sum_{a_n^1\in A_n^1(i_n)}\cdots\sum_{a_n^N\in A^N_n(i_n)}\sum_{i_{n+1}\in S_{n+1}\setminus D}\mathbb{P}_{\gamma}^{\bm{\pi}}
\left(\bigcap_{m=n+2}^{k+n+1}
\{X_m\in S_m\setminus D\}\Big{|} H_{n+1}=h_{n+1}\right)\\
&&~\times {p}_n(i_{n+1}\mid i_n,\bm{a}_n)
\pi_n^1(a_n^1\mid h_n)\cdots \pi_n^N(a^N_n\mid h_n)\\
&\leq&\sum_{a_n^1\in A_n^1(i_n)}\cdots\sum_{a_n^N\in A^N_n(i_n)}\sum_{i_{n+1}\in S_{n+1}\setminus D} (1-\beta_{n+1})(1-\beta_{n+2})\cdots(1-\beta_{n+k}){p}_n(i_{n+1}\mid i_n,\bm{a}_n)\\
&&~
\pi_n^1(a_n^1\mid h_n)\cdots \pi_n^N(a^N_n\mid h_n)\\
&\leq&(1-\beta_n)(1-\beta_{n+1})\cdots(1-\beta_{n+k}),
\end{eqnarray*}
where
$\bm{r}_n(i_n,\bm{a}_n):=(r^1_n(i_n,\bm{a}_n), \ldots, r^N_n(i_n,\bm{a}_n))$.
Hence, by induction we obtain (\ref{induction1}), which implies for any fixed $n\ge0$,
\begin{eqnarray*}
1\geq \mathbb{P}_{\gamma}^{\bm{\pi}}(\tau_D^n<\infty\mid H_n=h_n)
&=&1-\lim_{k\to\infty}\mathbb{P}_{\gamma}^{\bm{\pi}}(\tau_D^n>n+k\mid H_n=h_n)\\
&\ge&1- \lim_{k\to\infty}(1-\beta_n)(1-\beta_{n+1})\cdots(1-\beta_{n+k})\\
&=&1-\prod_{i=n}^{\infty}(1-\beta_i)=1,
\end{eqnarray*}
 where the last equality is due to the fact that $\sum_{n=0}^{+\infty}\beta_n=+\infty$ and $1>\beta_n\geq0$ for all $n\ge 0$.
Thus we can derive Assumption A.
\end{proof}

 \begin{remark} Obviously, the condition in Proposition 1 for the nonstationary case  is an extension of  Condition 1 in \cite{Huang2020} for the stationary model. Moreover, we will give an example in Section 5 (see Remark \ref{rem3}), for which the condition in Proposition 1  is  satisfied, but for which those  in \cite{Huang2020}  fail to hold.
\end{remark}

To state {\color{black}the} comparison theorem, let $\mathscr{P}(U)$ be the family of probability measures on any Borel set $U$ endowed with the weak topology, and for each $k\in I$, let $\mathscr{U}_n^k$ ($n\ge0$) be the set of functions $u: (S_n\setminus D)\times \mathbb{Q}^N\to[0,1]$, such that 
$u(i,\bm{\lambda})=1$ for each $i\in S_n\setminus D$ if $\bm{\lambda}=(\lambda^1,\ldots,\lambda^N)$ with $\lambda^k<0$.

In addition, for each $k\in I$,  $n\ge 0$, $\bm{\phi}:=(\phi^1,\ldots,\phi^N)\in\times_{k=1}^N\mathscr{P}(A_n^k(i))$ , we introduce the operators $T_{n,k}^{\bm{\phi}}$ from $\mathscr{U}_{n+1}^k$ to $\mathscr{U}_n^k$ as follows. For $i\in S_n\setminus D$,   and $u\in \mathscr{U}_{n+1}^k$,
\begin{eqnarray}
T_{n,k}^{\bm{a}}u(i,\bm{\lambda})&:=&\mathbb{I}_{[\lambda^k,\infty)}(r^k_n(i,\bm{a}))p_n(D\mid i,\bm{a})+\sum_{j\in S_{n+1}\setminus D}u(j,\bm{\lambda}-\bm{r}_n(i,\bm{a}))p_n(j\mid i,\bm{a}), \ \ \ \label{T1}\\
T_{n,k}^{\bm{\phi}}u(i,\bm{\lambda})&:=&\sum_{a^1\in A_n^1(i)}\ldots\sum_{a^N\in A_n^N(i)}T_{n,k}^{\bm{a}}u(i,\bm{\lambda})\phi^1(a^1)\ldots\phi^N(a^N). \label{T2}
\end{eqnarray}
 Obviously, $T_{n,k}^{\bm{\phi}}u(i,\bm{\lambda})\equiv 1$ when $\lambda^k<0$.

Now we give {\color{black}the} comparison theorem for the probability criterion of the nonstationary nonzero-sum games as follows.
\begin{theorem}\label{thm1}(Comparison Theorem). Under Assumption A,
the following assertions hold.
\begin{description}
\item[(a)]
For any fixed $k\in I$ and policy  $\bm{\pi}=(\pi^1,\pi^2,\ldots,\pi^N)\in \bm{\Pi}$ with $\pi^k= \{\pi_n^k,n\ge0\}$, if a sequence $\{u_n,n\geq 0\}$ of functions $u_n\in\mathscr{U}_n^k$ {satisfies}
\begin{eqnarray}\label{thm1.3}
u_n(i,\bm{\lambda})\leq(\ge)
T_{n,k}^{\bm{\pi}_n(\cdot\mid h_n)}u_{n+1}(i,\bm{\lambda}) \ \ \forall (i,\bm{\lambda})\in (S_{n}\setminus D)\times \mathbb{Q}^N,~h_n\in \mathbb{H}_n, ~n\ge 0,
\end{eqnarray}
where $\bm{\pi}_n(\cdot\mid h_n):=(\pi_n^1(\cdot\mid h_n),\ldots,\pi_n^N(\cdot\mid h_n))$, then
 \begin{eqnarray*}
&&u_0(i,\bm{\lambda})\leq(\ge) F_0^{k}(i,\bm{\lambda},\bm{\pi})~~\forall~(i,\bm{\lambda})\in (S_{0}\setminus D)\times\mathbb{Q}^N.
\end{eqnarray*}
\item[(b)]  For any fixed $k\in I$ and policy  $\bm{\pi}=(\pi^1,\pi^2,\ldots,\pi^N)\in \bm{\Pi}_m$ with $\pi^k= \{\phi_n^k,n\ge0\}$, if a sequence $\{u_n,n\geq 0\}$ of functions $u_n\in\mathscr{U}_n^k$ {satisfies}
\begin{eqnarray}
	u_n(i,\bm{\lambda})\leq(\ge)
	T_{n,k}^{\bm{\phi}_n(\cdot\mid i,\bm{\lambda})}u_{n+1}(i,\bm{\lambda}) \ \ \forall (i,\bm{\lambda})\in (S_{n}\setminus D)\times \mathbb{Q}^N,~h_n\in \mathbb{H}_n, ~n\ge 0,
\end{eqnarray}
where $\bm{\phi}_n(\cdot\mid i,\bm{\lambda}):=(\phi_n^1(\cdot\mid i,\bm{\lambda}),\ldots,\phi_n^N(\cdot\mid i,\bm{\lambda}))$ , then
\begin{eqnarray*}
	&&u_n(i,\bm{\lambda})\leq(\ge) F_n^{k}(i,\bm{\lambda},\bm{\pi})~~\forall~(i,\bm{\lambda})\in (S_n\setminus D)\times\mathbb{Q}^N~\text{and}~n\ge0.
\end{eqnarray*}
\item[(c)] For any fixed $k\in I$ and policy  $\bm{\pi}=(\pi^1,\pi^2,\ldots,\pi^N)\in \bm{\Pi}_m$ with $\pi^k= \{\phi_n^k,n\ge0\}$, the sequence  $\{F_n^k(i,\bm{\lambda},\bm{\pi}),n\ge 0\}$ is the {\color{black}unique solution of:}
\begin{eqnarray}\label{thm1.1}
	u_n(i,\bm{\lambda})=
T_{n,k}^{\bm{\phi}_n(\cdot\mid i,\bm{\lambda})}u_{n+1}(i,\bm{\lambda})~~\forall~(i,\bm{\lambda})\in (S_n\setminus D)\times\mathbb{Q}^N~\text{and}~n\ge0.
\end{eqnarray}
\end{description}
\end{theorem}
\begin{proof}
	(a).  For any fixed $k\in I$ and policy $\bm{\pi}\in \bm{\Pi}$,
	we first prove by induction that
	\begin{eqnarray}\label{induction3}
		u_0(X_0,\bm{\Lambda}_0)\mathbb{I}_{\{X_0\in S_0\setminus D\}}&\leq&
		\mathbb{E}_{\gamma}^{\bm{\pi}}
		[\mathbb{I}_{(-\infty,0]}(\Lambda^k_0-\sum_{m=0}^{\tau_{D}-1}
		r^k_m(X_m,Y_m))
		\mathbb{I}_{\{X_0\in S_0\setminus D\}}\mathbb{I}_{\{\tau_D<n+1\}}\mid X_0,\bm{\Lambda}_0]\nonumber\\
		&&+\mathbb{E}_{\gamma}^{\bm{\pi}}
		[u_n(X_n,\bm{\Lambda}_n)
		\mathbb{I}_{\{\cap_{m=0}^{n}\{X_m\in S_m\setminus D\}\}}\mid X_0,\bm{\Lambda}_0]~~\forall~n\ge 1,
	\end{eqnarray}
	where $\bm{\Lambda}_{m+1}=(\Lambda_{m+1}^1,\ldots,\Lambda_{m+1}^N):=\bm{\Lambda}_m-\bm{r}_m(X_m,Y_m)$   for any $m\ge 0$.
	
For $n=1$, by (\ref{thm1.3}) we have
	\begin{eqnarray}\label{iterating}
		&&u_0(X_0,\bm{\Lambda}_0)\mathbb{I}_{\{X_0\in S_0\setminus D\}}\nonumber\\
		&\leq&
		\sum_{a^1_0\in A^1_0(X_0)}\ldots\sum_{a^N_0\in A^N_0(X_0)}\bigg[\mathbb{I}_{[\Lambda_0,\infty)}(r^k_0(X_0,\bm{a}_0))p_0(D\mid X_0,\bm{a}_0)+\sum_{i_1\in S_{1}\setminus D}u_1(i_1,\bm{\Lambda}_0-\bm{r}_0(X_0,\bm{a}_0))\nonumber\\
		&&\times p_0(i_1\mid X_0,\bm{a}_0)\bigg]\pi_0^1(a_0^1\mid X_0,\bm{\Lambda}_0)\ldots\pi^N_0(a_0^N\mid X_0,\bm{\Lambda}_0)
		\mathbb{I}_{\{X_0\in S_0\setminus D\}}\nonumber\\
		&=&\mathbb{E}_{\gamma}^{\bm{\pi}}
		[\mathbb{I}_{(-\infty,0]}(\Lambda_1^k)\mathbb{I}_{\{X_0\in S_0\setminus D\}}
		\mathbb{I}_{\{\tau_D=1\}}
		+u_1(X_1,\bm{\Lambda}_1)\mathbb{I}_{\{X_0\in S_0\setminus D\}}\mathbb{I}_{\{X_1\in S_1\setminus D\}}
		\mid X_0,\bm{\Lambda}_0]\nonumber\\
		&=&\mathbb{E}_{\gamma}^{\bm{\pi}} [\mathbb{I}_{(-\infty,0]}(\Lambda_0^k-\sum_{m=0}^{\tau_{D}-1}
		r^k_m(X_m,Y_m))\mathbb{I}_{\{X_0\in S_0\setminus D\}}
		\mathbb{I}_{\{\tau_D<2\}}\mid X_0,\bm{\Lambda}_0]\nonumber\\
		&&+\mathbb{E}_{\gamma}^{\bm{\pi}}
		[u_1(X_1,\bm{\Lambda}_1)\mathbb{I}_{\{\cap_{m=0}^1\{X_m\in (S_m\setminus D)\}\}}
		\mid X_0,\bm{\Lambda}_0].
	\end{eqnarray}
	Now we assume that (\ref{induction3}) holds for some $T\ge 1$.
	\\Then, for $T+1$, by (\ref{thm1.3})  and the induction hypothesis, we have
	\begin{eqnarray*}
		&&u_0(X_0,\bm{\Lambda}_0)\mathbb{I}_{\{X_0\in S_0\setminus D\}}\\
		&\leq&
		\mathbb{E}_{\gamma}^{\bm{\pi}}
		[\mathbb{I}_{(-\infty,0]}(\Lambda_0^k-\sum_{m=0}^{\tau_{D}-1}
		r^k_m(X_m,Y_m))
		\mathbb{I}_{\{X_0\in S_0\setminus D\}}\mathbb{I}_{\{\tau_D<T+1\}}\mid X_0,\bm{\Lambda}_0]\nonumber\\
		&&+\mathbb{E}_{\gamma}^{\bm{\pi}}
		[u_T(X_T,\bm{\Lambda}_T)
		\mathbb{I}_{\{\cap_{m=0}^T\{X_m\in S_m\setminus D\}\}}\mid X_0,\bm{\Lambda}_0]\\
		&=&
		\mathbb{E}_{\gamma}^{\bm{\pi}} [\mathbb{I}_{(-\infty,0]}(\Lambda_0^k-\sum_{m=0}^{\tau_{D}-1}
		r^k_m(X_m,Y_m))
		\mathbb{I}_{\{X_0\in S_0\setminus D\}}
		\mathbb{I}_{\{\tau_D<T+1\}}\mid X_0,\bm{\Lambda}_0]\nonumber\\
		&&+\mathbb{E}_{\gamma}^{\bm{\pi}}
		[\mathbb{E}_{\gamma}^{\bm{\pi}}
		[u_T(X_T,\bm{\Lambda}_T)
		\mathbb{I}_{\{\cap_{m=0}^T\{X_m\in S_m\setminus D\}\}}\mid X_0,\bm{\Lambda}_0,Y_0,\ldots,X_T,\bm{\Lambda}_T]\mid X_0,\bm{\Lambda}_0]\\
		&\leq&
		\mathbb{E}_{\gamma}^{\bm{\pi}} [\mathbb{I}_{(-\infty,0]}(\Lambda_0^k-\sum_{m=0}^{\tau_{D}-1}
		r^k_m(X_m,Y_m))
		\mathbb{I}_{\{X_0\in S_0\setminus D\}}\mathbb{I}_{\{\tau_D<T+1\}}\mid X_0,\bm{\Lambda}_0]\nonumber\\
		&&+\mathbb{E}_{\gamma}^{\bm{\pi}}
		\bigg[
		\sum_{a_T^1\in A_T^1(X_T)}\ldots\sum_{a_T^N\in A_T^N(X_T)}\bigg\{\mathbb{I}_{(-\infty,0]}(\Lambda_T^k-r^k_T(X_T,\bm{a}_T))p_T(D\mid X_T,\bm{a}_T)\nonumber\\
		&&
		+\sum_{i_{T+1}\in S_{T+1}\setminus D}u_{T+1}(i_{T+1},\bm{\Lambda}_T-\bm{r}_T(X_T,\bm{a}_T))p_T(i_{T+1}\mid X_T,\bm{a}_T)\bigg\}
		\nonumber\\
		&&\times\pi_T^1(a_T^1\mid X_0,\bm{\Lambda}_0,Y_0,\ldots,X_T,\bm{\Lambda}_T)
		\ldots\pi_T^N(a_T^N\mid X_0,\bm{\Lambda}_0,Y_0,\ldots,X_T,\bm{\Lambda}_T)\nonumber\\
		&&
		\times\mathbb{I}_{\{\cap_{m=0}^T\{X_m\in S_m\setminus D\}\}}\bigg{|} X_0,\bm{\Lambda}_0\bigg]\\
		&=&\mathbb{E}_{\gamma}^{\bm{\pi}} [\mathbb{I}_{(-\infty,0]}(\Lambda_0^k-\sum_{m=0}^{\tau_{D}-1}r^k_m(X_m,Y_m))
		\mathbb{I}_{\{X_0\in S_0\setminus D\}}\mathbb{I}_{\{\tau_D<T+1\}}\mid X_0,\bm{\Lambda}_0]\nonumber\\
		&&+\mathbb{E}_{\gamma}^{\bm{\pi}}
		[\mathbb{E}_{\gamma}^{\bm{\pi}}
		[\mathbb{I}_{(-\infty,0]}(\Lambda^k_{T+1})
		\mathbb{I}_{\{\cap_{m=0}^T\{X_m\in S_m\setminus D\}\}}
		\mathbb{I}_{\{\tau_D=T+1\}}\mid X_0,\bm{\Lambda}_0,Y_0,\ldots,X_T,\bm{\Lambda}_T]\\
		&&+\mathbb{E}_{\gamma}^{\bm{\pi}}
		[u_{T+1}(X_{T+1},\bm{\Lambda}_{T+1})
		\mathbb{I}_{\{\cap_{m=0}^{T+1}\{X_m\in S_m\setminus D\}\}}
		\mid X_0,\bm{\Lambda}_0,Y_0,\ldots,X_T,\bm{\Lambda}_T]
		\mid X_0,\bm{\Lambda}_0]\\
		&=&\mathbb{E}_{\gamma}^{\bm{\pi}} [\mathbb{I}_{(-\infty,0]}(\Lambda_0^k-\sum_{m=0}^{\tau_{D}-1}
		r^k_m(X_m,Y_m))
		\mathbb{I}_{\{X_0\in S_0\setminus D\}}
		\mathbb{I}_{\{\tau_D<T+1\}}\mid X_0,\bm{\Lambda}_0]\nonumber\\
		&&+\mathbb{E}_{\gamma}^{\bm{\pi}}
		[\mathbb{I}_{(-\infty,0]}(\Lambda_0^k-\sum_{m=0}^{\tau_{D}-1}
		r^k_m(X_m,Y_m))
		\mathbb{I}_{\{X_0\in S_0\setminus D\}}\mathbb{I}_{\{\tau_D=T+1\}}
		\mid X_0,\bm{\Lambda}_0]\nonumber\\
		&&+\mathbb{E}_{\gamma}^{\bm{\pi}}
		[u_{T+1}(X_{T+1},\bm{\Lambda}_{T+1})
		\mathbb{I}_{\{\cap_{m=0}^{T+1}\{X_m\in S_m\setminus D\}\}}\mid X_0,\bm{\Lambda}_0]\nonumber\\
		&=&\mathbb{E}_{\gamma}^{\bm{\pi}} [\mathbb{I}_{(-\infty,0]}(\Lambda_0^k-\sum_{m=0}^{\tau_{D}-1}
		r^k_m(X_m,Y_m))
		\mathbb{I}_{\{X_0\in S_0\setminus D\}}
		\mathbb{I}_{\{\tau_D<T+2\}}\mid X_0,\bm{\Lambda}_0]\nonumber\\
		&&+\mathbb{E}_{\gamma}^{\bm{\pi}}
		[u_{T+1}(X_{T+1},\bm{\Lambda}_{T+1})
		\mathbb{I}_{\{\cap_{m=0}^{T+1}\{X_m\in S_m\setminus D\}\}}\mid X_0,\bm{\Lambda}_0].
	\end{eqnarray*}
	Hence, by induction it can be proved that (\ref{induction3}) holds for all $n\ge 1$. Thus,  by Assumption A and (\ref{induction3}), it can be derived as $n\to \infty$ that
	\begin{eqnarray*}
		u_0(X_0,\bm{\Lambda}_0)\mathbb{I}_{\{X_0\in S_0\setminus D\}}&\leq&
		\lim_{n\to\infty}(\mathbb{E}_{\gamma}^{\bm{\pi}}
		[\mathbb{I}_{(-\infty,0]}(\Lambda_0^k-\sum_{m=0}^{\tau_{D}-1}
		r^k_m(X_m,Y_m))
		\mathbb{I}_{\{X_0\in S_0\setminus D\}}\mathbb{I}_{\{\tau_D<n+1\}}\mid X_0,\bm{\Lambda}_0]\nonumber\\
		&&+\mathbb{E}_{\gamma}^{\bm{\pi}}
		[u_n(X_n,\bm{\Lambda}_n)
		\mathbb{I}_{\{\tau_D\ge n\}}\mid X_0,\bm{\Lambda}_0])\nonumber\\
		&=&\mathbb{P}_{\gamma}^{\bm{\pi}}
		\left(\sum_{m=0}^{\tau_{D}-1}r^k_m(X_m,Y_m)\ge \Lambda_0^k
		\Big{|} X_0,\bm{\Lambda}_0\right)\times \mathbb{I}_{\{X_0\in S_0\setminus D\}}\nonumber\\
		&=&F_0^k(X_0,\bm{\Lambda}_0,\bm{\pi})\mathbb{I}_{\{X_0\in S_0\setminus D\}}.
	\end{eqnarray*}
	Therefore, we obtain $u_0(i,\bm{\lambda})\leq F_0^k(i,\bm{\lambda},\bm{\pi})$ for any policy $\bm{\pi}\in \bm{\Pi}$ and $(i,\bm{\lambda})\in (S_0\setminus D)\times \mathbb{Q}^N$.

	The inverse case of ``$\ge$'' can be proved similarly.

(b). Under the conditions in (a), using the similar technique as above and the properties of Markov policies, we obtain
\begin{eqnarray*}
	u_n(i,\bm{\lambda})&\leq&
	\mathbb{E}_{\gamma}^{\bm{\pi}}
	[\mathbb{I}_{(-\infty,0]}(\Lambda^k_n-\sum_{m=n}^{\tau_{D}^n-1}
	r^k_m(X_m,Y_m))
	\mathbb{I}_{\{\tau_D^n<K+1\}}\mid X_n=i,\bm{\Lambda}_n=\bm{\lambda}]\\
	&&+\mathbb{E}_{\gamma}^{\bm{\pi}}
	[u_{K}(X_{K},\bm{\Lambda}_{K})
	\mathbb{I}_{\{\tau_D^n\ge K+1\}}\mid X_n=i,\bm{\Lambda}_n=\bm{\lambda}]~~\forall~K\ge n+1.
\end{eqnarray*}
Letting $K\to\infty$, together with the similar proof above, we get $u_n(i,\bm{\lambda})\leq F_n^k(i,\bm{\lambda},\bm{\pi})$ for all $\bm{\pi}\in \bm{\Pi}_m$.
Thus, we have proved part (b).

(c). For every $n\ge0$,  $\pi^1=\{\phi_m,k\in\mathbb{N}_0\}\in \Pi^1_m$, $\pi^2=\{\psi_m,k\in \mathbb{N}_0\}\in \Pi^2_m$ and $(i,\bm{\lambda})\in (S_{n}\setminus D)\times \mathbb{Q}$, we can derive that
\begin{eqnarray*}
	&&F_n^{k}(i,\bm{\lambda},\bm{\pi})\\
	&=&\mathbb{P}_{\gamma}^{\bm{\pi}}
	\left(\sum_{l=n}^{\tau_D^n-1}r^k_l(X_l,Y_l)
	\geq \lambda^k \Big{|} X_n=i,\bm{\Lambda}_n=\bm{\lambda}\right)\\
	&=&\mathbb{P}_{\gamma}^{\bm{\pi}}
	\left(\sum_{l=n+1}^{\tau_D^n-1}
	r^k_l(X_l,Y_l)
	\geq \lambda^k-r^k_n(X_n,Y_n), X_{n+1}\in D \Big{|}
	X_n=i,\bm{\Lambda}_n=\bm{\lambda}\right)\\
	&&+\mathbb{P}_{\gamma}^{\bm{\pi}}
	\left(\sum_{l=n+1}^{\tau_D^n-1}
	r^k_l(X_l,Y_l)
	\geq \lambda^k-r^k_n(X_n,Y_n), X_{n+1}\in S_{n+1}\setminus D  \Big{|}
	X_n=i,\bm{\Lambda}_n=\bm{\lambda}\right)\\
	&=&\mathbb{P}_{\gamma}^{\bm{\pi}}
	\left(0\geq \lambda^k-r^k_n(X_n,Y_n), X_{n+1}\in D \Big{|}
	X_n=i,\bm{\Lambda}_n=\bm{\lambda}\right)\\
	&&+\sum_{j\in S_{n+1}\setminus D} \mathbb{P}_{\gamma}^{\bm{\pi}}
	\left(\sum_{l=n+1}^{\tau_D^n-1}
	r^k_l(X_l,Y_l)
	\geq \lambda^k-r^k_n(X_n,Y_n), X_{n+1}=j  \Big{|}
	X_n=i,\bm{\Lambda}_n=\bm{\lambda}\right)\\
	&=&\sum_{a^1_n\in A_n^1(i)}\ldots\sum_{a^N_n\in A_n^N(i)}
	\bigg\{\mathbb{I}_{[\lambda^k,\infty)}
	(r^k_n(i,\bm{a}_n)){p}_n(D\mid i,\bm{a}_n)\\
	&&+
	\sum_{j\in S_{n+1}\setminus D}
	\mathbb{P}_{\gamma}^{\bm{\pi}}
	\left(\sum_{l=n+1}^{\tau_D^{n+1}-1}
	r^k_l(X_l,Y_l)
	\geq \lambda^k-r^k_n(i,\bm{a}_n)\Big{|}X_{n+1}=j,
	\bm{\lambda}_{n+1}=\bm{\lambda}-\bm{r}_n(i,\bm{a}_n)\Bigg)\right.\\
	&&~~~ \times
	p_n(j\mid i,\bm{a}_n)\bigg\}
	\phi^1_n(a^1_n\mid i,\bm{\lambda})\ldots\phi^N_n(a^N_n\mid i,\bm{\lambda})\\
	&=&\sum_{a^1_n\in A_n^1(i)}\ldots\sum_{a^N_n\in A_n^N(i)}
	\bigg\{\mathbb{I}_{[\lambda^k,\infty)}
	({r}^k_n(i,\bm{a}_n)){p}_n(D\mid i,\bm{a}_n)\\
	&&+
	\sum_{j\in S_{n+1}\setminus D}
	F_{n+1}^{k}(j,\bm{\lambda}-\bm{r}_n(i,\bm{a}_n),\bm{\pi})
	p_n(j\mid i,\bm{a}_n)\bigg\}
		\phi^1_n(a^1_n\mid i,\bm{\lambda})\ldots\phi^N_n(a^N_n\mid i,\bm{\lambda})\\
	&=&
	T_{n,k}^{\bm{\phi}_n(\cdot \mid i,\bm{\lambda})}F_{n+1}^{k}(i,\bm{\lambda},\bm{\pi}),
\end{eqnarray*}
which verifies  (\ref{thm1.1}). Moreover, the uniqueness follows from Theorem 1(b).
\end{proof}

To consider the continuity of $F_{n}^{k}(i,\bm{\lambda},\bm{\pi})$ on $\bm{\Pi}_m$, we recall the definition of any sequence's convergence in $\Pi_m^k$ for any $k\in I$. Since $S_n$ and $\mathbb{Q}$ are denumerable  and $A_n^k(i)$ are finite, the set $\Pi_m^k$ is compact, and a sequence of $\pi^k_t=\{\phi_{n,t}^k,n\in \mathbb{N}_0\}(t\ge 1)$ in $\Pi_m^k$  {\color{black}converges to} a policy $\pi^k=\{\phi_n^k,n\in \mathbb{N}_0\}$ in $\Pi_m^k$ (as $t\to\infty$) if and only if $\lim_{t\to \infty}\phi_{n,t}^k(a\mid i,\bm{\lambda})=\phi_n^k(a\mid i,\bm{\lambda})$ for $a\in A_n^k(i)$, $(i,\bm{\lambda})\in S_n\times \mathbb{Q}^N$ and $n\ge 0$.

\begin{lemma}
	Under Assumption A, for any fixed $(i,\bm{\lambda})\in S_n\times \mathbb{Q}^N$ and $k\in I$, $F_{n}^{k}(i,\bm{\lambda},\bm{\pi})$ is uniformly continuous in  $\bm{\pi}\in\bm{\Pi}_m$.
	
	\begin{proof}  Given any  $k\in I$, since $\bm{\Pi}_m$ is   compact, it suffices to show the continuity of $F_{n}^{k}(i,\bm{\lambda},\bm{\pi})$ on $\bm{\Pi}_m$.
		Let $\pi^k_t=\{\phi_{n,t}^k,n\in \mathbb{N}_0\}(t\ge 1)$ in $\Pi_m^k$ such that $\pi^k_t\to\pi^k=\{\phi_{n}^k,n\in \mathbb{N}_0\}$ in $\Pi_m^k$, as $t\to \infty$. Then, for each $(i,\bm{\lambda})\in S_n\times \mathbb{Q}^N$, $n\in \mathbb{N}_0$, $u\in\mathscr{U}_{n+1}^k$, by (\ref{T1}) and (\ref{T2}), we have
		\begin{eqnarray}\label{lem1(1)}
			\lim_{t\to\infty}	T_{n,k}^{\bm{\phi}_{n,t}(\cdot \mid i,\bm{\lambda})}u(i,\bm{\lambda})=T_{n,k}^{\bm{\phi}_{n}(\cdot \mid i,\bm{\lambda})}u(i,\bm{\lambda}),
			\end{eqnarray}
		and (by Theorem 1(c))
		\begin{eqnarray}\label{lem1(2)}
			F_n^{k}(i,\bm{\lambda},\bm{\pi}_t)=
			T_{n,k}^{\bm{\phi}_{n,t}(\cdot \mid i,\bm{\lambda})}F_{n+1}^{k}(i,\bm{\lambda},\bm{\pi}_t), ~n\in\mathbb{N}_0,~t\ge1.
		\end{eqnarray}
	Thus, by (\ref{lem1(1)})-(\ref{lem1(2)}) and the extension of Fatou's Lemma (i.e. Lemma 8.3.7 in \cite{Hernandez1999}), we have
	\begin{eqnarray*}
		&&\limsup_{t\to\infty}F_n^{k}(i,\bm{\lambda},\bm{\pi}_t)\leq
		T_{n,k}^{\bm{\phi}_{n}(\cdot \mid i,\bm{\lambda})}\limsup_{t\to\infty}F_{n+1}^{k}(i,\bm{\lambda},\bm{\pi}_t),~n\in\mathbb{N}_0,\\
		&&\liminf_{t\to\infty}F_n^{k}(i,\bm{\lambda},\bm{\pi}_t)\ge
		 T_{n,k}^{\bm{\phi}_{n}(\cdot \mid i,\bm{\lambda})}\liminf_{t\to\infty}F_{n+1}^{k}(i,\bm{\lambda},\bm{\pi}_t),~n\in\mathbb{N}_0,
	\end{eqnarray*}
which, together with Theorem 1(b), implies
\begin{eqnarray*}
	\limsup_{t\to\infty}F_n^{k}(i,\bm{\lambda},\bm{\pi}_t)\leq F_n^{k}(i,\bm{\lambda},\bm{\pi}), ~~\liminf_{t\to\infty}F_n^{k}(i,\bm{\lambda},\bm{\pi}_t)\ge F_n^{k}(i,\bm{\lambda},\bm{\pi}),
	\end{eqnarray*}
and so $
	\limsup_{t\to\infty}F_n^{k}(i,\bm{\lambda},\bm{\pi}_t)=\liminf_{t\to\infty}F_n^{k}(i,\bm{\lambda},\bm{\pi}_t)= F_n^{k}(i,\bm{\lambda},\bm{\pi}).$
	\end{proof}
\end{lemma}

\section{The existence of a Nash equilibrium}\label{sec4}

In this section, we establish the existence of a Nash  equilibrium. To do so, we need to introduce some notation below.

Given any $\bm{\pi}\in\bm{\Pi}_m$, $n\ge1, k\in I,$ we define a function ${\bm\hat F}_n^{k}(\bm{\pi}^{-k})$ on $S_n\times \mathbb{Q}^N$ as follows:
\begin{eqnarray} {\bm\hat F}_n^{k}(i,\bm{\lambda},\bm{\pi}^{-k}):=\sup_{\pi^k\in\Pi^k_m}F_n^{k}(i,\bm{\lambda},[\bm{\pi}^{-k},\pi^k]) \ \ \ \forall~(i,\bm{\lambda})\in S_n\times \mathbb{Q}^N.
\end{eqnarray}
Moreover, we define ${\bm\hat F}_0^{k}(i,\bm{\lambda},\bm{\pi}^{-k}):=\sup_{\pi^k\in\Pi^k}F_0^{k}(i,\bm{\lambda},[\bm{\pi}^{-k},\pi^k])~ \forall~(i,\bm{\lambda})\in S_0\times \mathbb{Q}^N.
$
\begin{theorem}\label{thm2}
Under  Assumption A, the following assertions hold.
\begin{description}
\item[(a)] Given any fixed $k\in I$, the sequence  $\{	 {\bm\hat F}_n^{k}(i,\bm{\lambda},\bm{\pi}^{-k}),n\ge 0\}$ uniquely solves  the following equations in $\{\mathscr{U}^k_n,n\ge0\}$:
\begin{eqnarray}\label{optimality}
	u_n(i,\bm{\lambda})=\sup_{\phi\in\mathscr{P}(A_n^k(i))}
	T_{n,k}^{[\bm{\phi}_n^{-k}(\cdot|i,\bm{\lambda}),\phi]}u_{n+1}(i,\bm{\lambda})~~\forall (i,\bm{\lambda})\in (S_n\setminus D)\times\mathbb{Q}^N,~n\ge0,
\end{eqnarray}
and there exists $\{\phi_{n*}^{k},n\ge0\}\in \Pi_m^k$ such that
\begin{eqnarray}\label{optimality1}
 {\bm\hat F}_n^{k}(i,\bm{\lambda},\bm{\pi}^{-k})=
	T_{n,k}^{[\bm{\phi}_n^{-k}(\cdot|i,\bm{\lambda}),\phi_{n*}^{k}(\cdot|i,\bm{\lambda})]} {\bm\hat F}_{n+1}^{k}(i,\bm{\lambda},\bm{\pi}^{-k})~~\forall (i,\bm{\lambda})\in (S_n\setminus D)\times\mathbb{Q}^N,n\geq 0.
\end{eqnarray}
\item[(b)] There exists  $\bm{\pi}_*=\{\pi_*^k,k\in I\}\in\bm{\Pi}_m$ where $\pi^{k}_*=\{\psi_{n,*}^{k},n\ge0\}\in\Pi_m^k$, such that
\begin{eqnarray}\label{optimality2}
	F_n^{k}(i,\bm{\lambda},[\bm{\pi}_*^{-k},\pi^k_*])= {\bm\hat F}_n^{k}(i,\bm{\lambda},\bm{\pi}^{-k}_*)=\sup_{\phi\in \mathscr{P}(A_n^k(i))}	 T_{n,k}^{[\bm{\psi}_{n}^{-k}(\cdot|i,\bm{\lambda}),\phi]}\hat F_{n+1}^{k}(i,\bm{\lambda},\bm{\pi}_*^{-k})
\end{eqnarray}
 for all $(i,\bm{\lambda})\in (S_n\setminus D)\times\mathbb{Q}^N, k\in I,n\ge0.$
 \item[(c)] The multipolicy $\bm{\pi}_*$ obtained in (b) is a Nash equilibrium.
\end{description}
\end{theorem}
\begin{proof}
(a). Given any fixed $k\in I$ and $n\ge0$, we define {\color{black}the} operator $T_{n,k}^{\bm{\phi}_{n}^{-k}}$ by
\begin{eqnarray}
T_{n,k}^{\bm{\phi}_{n}^{-k}}u(i,\bm{\lambda}):=\sup_{\phi\in\mathscr{P}(A_n^k(i))}	 T_{n,k}^{[\bm{\phi}_{n}^{-k}(\cdot|i,\bm{\lambda}),\phi]}u(i,\bm{\lambda})~~\forall~(i,\bm{\lambda})\in (S_n\setminus D)\times\mathbb{Q}^N, u\in \mathscr{U}^k_{n+1},
\end{eqnarray}
which is used to construct  {\color{black}the} sequence $\{u_n^m,m\ge0\}$ in $\mathscr{U}^k_{n}$ by
\begin{eqnarray}\label{algorithm1}
u_n^{0}(i,\bm{\lambda}):=\mathbb{I}_{(-\infty,0]}(\lambda^k)~~and ~~u_n^{m+1}(i,\bm{\lambda}):=T_{n,k}^{\bm{\phi}_{n}^{-k}}u_{n+1}^{m}(i,\bm{\lambda})~ ~\forall~m\ge 0.
\end{eqnarray}
Obviously, we have $u_n^m(\cdot,\cdot)\in\mathscr{U}_n^k$ for all $m\ge0$.
Next we prove the fact that
\begin{eqnarray}\label{2.5}
u_n^{m+1}(i,\bm{\lambda})\geq u_n^{m}(i,\bm{\lambda})~\forall (i,\bm{\lambda})\in (S_n\setminus D)\times\mathbb{Q}^N, m\ge0.
\end{eqnarray}
If $\lambda^k >0$,  then  $u_n^{0}(i,\bm{\lambda})=0$. So $u_n^{1}(i,\bm{\lambda})\geq u_n^{0}(i,\bm{\lambda})$, that is,  (\ref{2.5}) is  true for $m=0$ when $\lambda^k>0$.  If $\lambda^k\le0$,  by (\ref{algorithm1}) and $r_n^k\geq 0$, we have
\begin{eqnarray*}
&&u_n^1(i,\bm{\lambda})\nonumber\\
&=&\sup_{\phi\in\mathscr{P}(A_n^k(i))}
\bigg\{\sum_{a^1\in A_n^1(i)}\ldots\sum_{a^N\in A_n^N(i)}\bigg[\mathbb{I}_{[\lambda^k,\infty)}(r_n^k(i,\bm{a}))
p_n(D\mid i,\bm{a})\nonumber\\
&&+\sum_{j\in S_{n+1}\setminus D}u_{n+1}^{0}(j,\bm{\lambda}-\bm{r}_n(i,\bm{a}))p_n(j\mid i,\bm{a})\bigg]\phi(a^k)\prod_{l\in I,l\ne k}\phi_n^l(a^l)\bigg\}\nonumber\\
&=&\sup_{\phi\in\mathscr{P}(A_n^k(i))}
\bigg\{\sum_{a^1\in A_n^1(i)}\ldots\sum_{a^N\in A_n^N(i)}\bigg[p_n(D\mid i,\bm{a})+\sum_{j\in S_{n+1}\setminus D} p_n(j\mid i,\bm{a})\bigg]\phi(a^k)\prod_{l\in I,l\ne k}\phi_n^l(a^l)\bigg\}\nonumber\\
&=&1\geq u_n^{0}(i,\bm{\lambda}),
\end{eqnarray*}
Since $T_{n,k}^{\bm{\phi}_{n}^{-k}}$ is a monotone operator, we have  $u_n^{m+1}(i,\bm{\lambda})\geq u^{m}_n(i,\bm{\lambda})$ for all $(i,\bm{\lambda})\in (S_{n}\setminus D)\times\mathbb{Q^N}$ and $m\ge0$.
Thus, the limit  $\lim\limits_{m\to\infty}u_n^{m}(i,\bm{\lambda})
=:u_n^*(i,\bm{\lambda})$ exists. Since the limit of the measurable functions is measurable, we can obtain $u^*_n(i,\bm{\lambda})\in \mathscr{U}_n^k$.
Next we will prove that $\{u_n^*(i,\bm{\lambda}),n\ge 0\}$ satisfies {the following  equation:}
\begin{eqnarray}\label{2.6}
u_n^*(i,\bm{\lambda})=T_{n,k}^{\bm{\phi}_{n}^{-k}}u_{n+1}^*(i,\bm{\lambda})~~\forall~n\ge0.
\end{eqnarray}
By (\ref{2.5}) and $u_n^*(i,\bm{\lambda})=\lim\limits_{m\to\infty}u^{m}_n(i,\bm{\lambda})
$, we know the right side of (\ref{2.6})  {\color{black}satisfies}
\begin{eqnarray*}
T_{n,k}^{\bm{\phi}_{n}^{-k}}u_{n+1}^*(i,\bm{\lambda})\geq T_{n,k}^{\bm{\phi}_{n}^{-k}}u_{n+1}^{m}(i,\bm{\lambda})
=u_n^{m+1}(i,\bm{\lambda})~~\forall~(i,\bm{\lambda})\in (S_{n}\setminus D)\times \mathbb{Q}^N,~m\ge 0~and~n\ge0,
\end{eqnarray*}
which implies that
\begin{eqnarray}\label{2.7}
T_{n,k}^{\bm{\phi}_{n}^{-k}}u_{n+1}^*(i,\bm{\lambda})\geq u_n^*(i,\bm{\lambda})~~\forall~(i,\bm{\lambda})\in (S_{n}\setminus D)\times \mathbb{Q}^N,~n\ge0.
\end{eqnarray}
Now we show the reverse inequality. It follows from (\ref{algorithm1}) that
\begin{eqnarray*}
u_n^{m+1}(i,\bm{\lambda})
&=&
\sup_{\phi\in\mathscr{P}(A_n^k(i))}
\bigg\{{\color{black}\sum_{a^1\in A_n^1(i)}\ldots\sum_{a^N\in A_n^N(i)}}\bigg[\mathbb{I}_{[\lambda^k,\infty)}(r_n^k(i,\bm{a}))
p_n(D\mid i,\bm{a})\nonumber\\
&&+\sum_{j\in S_{n+1}\setminus D}u_{n+1}^{m}(j,\bm{\lambda}-\bm{r}_n(i,\bm{a}))p_n(j\mid i,\bm{a})\bigg]\phi(a^k)\prod_{l\in I,l\ne k}\phi_n^l(a^l)\bigg\}.
\end{eqnarray*}
{\color{black}Thus, 
for any $\phi\in\mathscr{P}(A_n^k(i))$, we have 
\begin{eqnarray}\label{2.8}
	u_n^{m+1}(i,\bm{\lambda})
	&\geq&\sum_{a^1\in A_n^1(i)}\ldots\sum_{a^N\in A_n^N(i)}\bigg[\mathbb{I}_{[\lambda^k,\infty)}(r_n^k(i,\bm{a}))
	p_n(D\mid i,\bm{a})\nonumber\\
	&&+\sum_{j\in S_{n+1}\setminus D}u_{n+1}^{m}(j,\bm{\lambda}-\bm{r}_n(i,\bm{a}))p_n(j\mid i,\bm{a})\bigg]\phi(a^k)\prod_{l\in I,l\ne k}\phi_n^l(a^l).
\end{eqnarray}
Taking $m\to\infty$ in (\ref{2.8}), by Proposition C.12 in \cite{Hernandez1996} we have
\begin{eqnarray*}
u_n^*(i,\bm{\lambda})
&\geq&\sum_{a^1\in A_n^1(i)}\ldots\sum_{a^N\in A_n^N(i)}\bigg[\mathbb{I}_{[\lambda^k,\infty)}(r_n^k(i,\bm{a}))
p_n(D\mid i,\bm{a})\nonumber\\
&&+\sum_{j\in S_{n+1}\setminus D}u_{n+1}^{*}(j,\bm{\lambda}-\bm{r}_n(i,\bm{a}))p_n(j\mid i,\bm{a})\bigg]\phi(a^k)\prod_{l\in I,l\ne k}\phi_n^l(a^l).
\end{eqnarray*}
Then we get
\begin{eqnarray}\label{2.9}
u_n^*(i,\bm{\lambda})
&\geq&\sup_{\phi\in\mathscr{P}(A_n^k(i))}\bigg\{\sum_{a^1\in A_n^1(i)}\ldots\sum_{a^N\in A_n^N(i)}\bigg[\mathbb{I}_{[\lambda^k,\infty)}(r_n^k(i,\bm{a}))
p_n(D\mid i,\bm{a})\nonumber\\
&&+\sum_{j\in S_{n+1}\setminus D}u_{n+1}^{*}(j,\bm{\lambda}-\bm{r}_n(i,\bm{a}))p_n(j\mid i,\bm{a})\bigg]\phi(a^k)\prod_{l\in I,l\ne k}\phi_n^l(a^l)\bigg\}\nonumber\\
&\geq& T_{n,k}^{\bm{\phi}_{n}^{-k}}u_{n+1}^*(i,\bm{\lambda}),
\end{eqnarray}}
which, together with (\ref{2.7}), gets (\ref{2.6}). So $\{u_n^*(i,\bm{\lambda}),n\ge 0\}$ satisfies (\ref{2.6}).

To prove  (a), it suffices to show $u_n^*(i,\bm{\lambda})={\bf\hat F}_n^{k}(i,\bm{\lambda},\bm{\pi}^{-k})$ for all $(i,\bm{\lambda})\in (S_{n}\setminus D)\times \mathbb{Q}^N$ and  $n\ge 0$.
For arbitrary ${\pi}^k=\{\psi^k_n,n\ge0\}\in\Pi_m^k$, by (\ref{2.6}) we have
\begin{eqnarray}\label{o1}
	u_n^*(i,\bm{\lambda})
= \sup_{\phi\in\mathscr{P}(A_n^k(i))}T_{n,k}^{[\bm{\phi}_{n}^{-k}(\cdot|i,\bm{\lambda}),\phi]}u_{n+1}^*(i,\bm{\lambda})
\geq T_{n,k}^{[\bm{\phi}_{n}^{-k}(\cdot|i,\bm{\lambda}),\psi^k_n(\cdot|i,\bm{\lambda})]}u_{n+1}^*(i,\bm{\lambda}),
\end{eqnarray}
together with Theorem 1(b) and the arbitrariness of ${\pi}^k$, gives
\begin{eqnarray}\label{o2}
	u_n^*(i,\bm{\lambda})
	\geq \sup_{{\pi}^k\in\Pi_m^k}F_{n}^{k}(i,\bm{\lambda},[\bm{\pi}^{-k},{\pi}^k])
	={\bf\hat F}_n^{k}(i,\bm{\lambda},\bm{\pi}^{-k}).
\end{eqnarray}
Similarly, by Theorem 1(a), we can get $u_0^*(i,\bm{\lambda})
\geq {\bf\hat F}_0^{k}(i,\bm{\lambda},\bm{\pi}^{-k}).$
\\On the other hand, by (\ref{2.6}), there exists an $\phi_{n*}^{k}\in \Phi_n^k$ such that
\begin{eqnarray}\label{o3}
	u_n^*(i,\bm{\lambda})
	= T_{n,k}^{[\bm{\phi}_{n}^{-k}(\cdot|i,\bm{\lambda}),\phi_{n*}^{k}(\cdot|i,\bm{\lambda})]}u_{n+1}^*(i,\bm{\lambda}) \ \forall (i,\bm{\lambda})\in (S_{n}\setminus D)\times \mathbb{Q}^N,n\ge 0.
\end{eqnarray}
Let $\pi^{k}_*=\{\phi_{n*}^{k},n\geq 0\}$,  $\bm{\pi}_*= [\bm{\pi}^{-k},{\pi}^k_*]\in\bm{\Pi}_m$. Then, by Theorem 1(c) and (\ref{o3}) we get
$$u_n^*(i,\bm{\lambda})=F_{n}^{k}(i,\bm{\lambda},\bm{\pi}_*)~\forall~(i,\bm{\lambda})\in (S_{n}\setminus D)\times \mathbb{Q}^N,~n\ge 0.$$
Hence, by (\ref{o2}) we have
$$u_n^*(i,\bm{\lambda})={\bf\hat F}_n^{k}(i,\bm{\lambda},\bm{\pi}^{-k})=\sup_{{\pi}^k\in\Pi_m^k}F_{n}^{k}(i,\bm{\lambda},[\bm{\pi}^{-k},{\pi}^k]),$$
which, together with (\ref{o3}), gives (a).

(b) Given $\bm{\pi}=(\pi^1,\pi^2,\ldots,\pi^N)\in\bm{\Pi}_m$ with $\pi^k=\{\phi_n^k,n\ge0\}, n\geq 0,k\in I, \bm{\phi_n}=(\phi_n^k,k\in I)$, we define {\color{black}the} operator $ T_{n,k}^{\bm{\phi}_n}$ on $\mathscr{U}_{n+1}^k$  by
\begin{eqnarray}
 T_{n,k}^{\bm{\phi_n}}u(i,\bm{\lambda)}:= T_{n,k}^{\bm{\phi}_n(\cdot|i,\bm{\lambda})}u(i,\bm{\lambda)}, (i,\bm{\lambda})\in (S_{n}\setminus D)\times \mathbb{Q}^N,  u\in \mathscr{U}_{n+1}^k,\label{T}
 \end{eqnarray}
which, together with   ${\bf\hat F}_{n+1}^{k}(\bm{\pi}^{-k})$, is used to introduce the following notation:
\begin{eqnarray}\label{f_k}
	f_k(\bm{\pi}^{-k}):=\{\{\psi_{n}^k\}\in\Pi_m^k\mid T_{n,k}^{[\bm{\phi}_n^{-k},\psi_{n}^{k}]}{\bf\hat F}_{n+1}^{k}(\bm{\pi}^{-k})=\sup_{\tilde\pi^k\in \bm{\Pi}_m^k}	 T_{n,k}^{[\bm{\phi}_n^{-k},\tilde\pi^k]}{\bf\hat F}_{n+1}^{k}(\bm{\pi}^{-k})\}
\end{eqnarray}
for each $\bm{\pi}^{-k}\in\bm{\Pi}^{-k}_m$.
It is obvious that $f_k(\bm{\pi}^{-k})$ is nonempty by Theorem 2(a). Moreover, since  $\bm{\Pi_m}^k$ is compact,  {by Lemma 1}, $f_k(\bm{\pi}^{-k})$ is compact.
On the other hand, by (\ref{T2}), the function
\begin{eqnarray}
	\pi^k:=\{\psi_n^k\} \longmapsto T_{n,k}^{[\bm{\phi}_n^{-k},\psi_{n}^{k}]}{\bf\hat F}_{n+1}^{k}(\bm{\pi}^{-k})
\end{eqnarray}
is linear and hence convex in $\pi^k=\{\psi_n^k\}\in f_k(\bm{\pi}^{-k})$ for any fixed $\bm{\pi}^{-k}\in \bm{\Pi}_m^{-k}$.

We now define the multifunction
\begin{eqnarray*}
	f:~\times_{k=1}^N\Pi^k_m~&\to&~ 2^{\times_{k=1}^N\Pi^k_m}\\
	\bm{\pi}~&\longmapsto&~f(\bm{\pi}):=\times_{k=1}^N f_k(\bm{\pi}^{-k}).
	\end{eqnarray*}
We next show that mapping $f$ is upper-semicontinuous. Suppose that:
\begin{description}
\item[(P.1)]  Any  $\bm{\pi}_t=(\pi_t^k,k\in I),\bm{\pi}_t^*=(\pi_t^{k*},k\in I)\in\bm{\Pi}_m$, for all $t\ge 0$;
\item[(P.2)]  $\lim_{t\to\infty}\bm{\pi}_t=\bm{\pi}\in\bm{\Pi}_m$;
$\lim_{t\to\infty}\bm{\pi}_t^{*}=\bm{\pi}^{*}\in\bm{\Pi}_m$;
\item[(P.3)] $\bm{\pi}_t^*\in f(\bm{\pi}_t)$ for all  $t\geq 1$.
\end{description}
We wish to show that $\bm{\pi}^{*}$ is in $f(\bm{\pi})$.
To do this, for each  fixed $k\in I$ and any $t\ge 0$, let $\pi_t^k=:\{\phi_{n,t}^k,n\in\mathbb{N}_0\}, \pi_t^{k*}=:\{\psi_{n,t}^{k},n\in\mathbb{N}_0\}$, and $\pi^{k*}=:\{\psi_n^k,n\in\mathbb{N}_0\}$. Then, for any $n\ge 0$ and $(i,\bm{\lambda})\in (S_{n}\setminus D)\times \mathbb{Q}^N, i\in I$, by (P.3) and (\ref{f_k}) we have
\begin{eqnarray}\label{f_nk}
	{\bf\hat{F}}_{n}^{k}(i,\bm{\lambda},\bm{\pi}^{-k}_t)=T_{n,k}^{[\bm{\phi}_{n,t}^{-k},\psi_{n,t}^{k}]}{\bf\hat F}_{n+1}^{k}(i,\bm{\lambda},\bm{\pi}^{-k}_t)~\forall~t,n\ge0.
\end{eqnarray}
Since $F_{n}^{k}(i,\bm{\lambda},\cdot)$ is uniformly continuous on the compact space $\bm{\Pi}_m$ (by Lemma 1),    ${\bf\hat{F}}_{n}^{k}(i,\bm{\lambda},\bm{\pi}^{-k})$ is also continuous in $\bm{\pi}^{-k}\in \bm{\Pi}_m^{-k}$.
Thus, using Proposition C.12 in \cite{Hernandez1996} and the definition of the operators, letting $t\to\infty$ in (\ref{f_nk}) we have
\begin{eqnarray*}\label{thm2.01}
	{\bf\hat{F}}_{n}^{k}(i,\bm{\lambda},\bm{\pi}^{-k})=T_{n,k}^{[\bm{\phi}_{n}^{-k},\psi_{n}^{k}]}{\bf\hat{F}}_{n+1}^{k}(i,\bm{\lambda},\bm{\pi}^{-k})~\forall~n\ge0,
\end{eqnarray*}
which means that $\pi^{k*}\in f_k(\bm{\pi}^{-k})$ for each $k\in I$, and so $\bm{\pi}^{*}$ is in $\times_{k=1}^Nf_k(\pi^{-k})$, that is, $\bm{\pi}^*\in f(\bm{\pi})$. Thus, the mapping $F$ is indeed upper-semicontinuous.

Then, it follows from Fan's fixed point theorem (see \cite{Guo2005})  that there exists $\pi_*^k:=\{\psi^{k}_{n,*}\}\in \Pi_m^k$ such that $\bm{\pi}_{*}=:(\pi_*^k,k\in I)\in \times_{k=1}^N f_k(\bm{\pi_*}^{-k})$. Therefore,  by (\ref{f_k}) we have
  \begin{eqnarray*}\label{thm2.02}
  	{\bf\hat{F}}_{n}^{k}(\bm{\pi_*}^{-k})=T_{n,k}^{[\bm{\psi}_{n}^{-k},\psi_{n}^{k}]}{\bf\hat{F}}_{n+1}^{k}(\bm{\pi_*}^{-k})~\ \ \ \forall n\ge0,k\in I.
  \end{eqnarray*}
By the uniqueness in Theorem 1(c), we have
$F_{n}^{k}(i,\bm{\lambda},[\bm{\pi}_*^{-k},\pi^k_*])={\bf\hat {F}}_{n}^{k}(i,\bm{\lambda},\bm{\pi_*}^{-k})~~\forall k\in I,n\geq 0,$
which, together with Theorem 2(a), proves  (b).

(c) By (b) and Theorem 1, we have
$	F_{0}^{k}(i,\bm{\lambda},[\bm{\pi}_*^{-k},\pi^k_*])={\hat{F}}_{0}^{k}(i,\bm{\lambda},\bm{\pi}_*^{-k})\ge F_{0}^{k}(i,\bm{\lambda},[\bm{\pi}_*^{-k},\pi^k])$ for all $\pi^k\in\Pi^k$,
which implies $\bm{\pi}_*$ is a Nash equilibrium.
\end{proof}

\section{An Algorithm}\label{sec5}
In this section, we will give another main result (see  Theorem \ref{thm3} below) on an algorithm to find $\epsilon$-Nash equilibria.
To do so, we need the following assumption.

{\bf{Assumption B.}} $\beta:=\inf_{i\in S_n, \bm{a}\in\bm{A}_n(i), n\ge 0}p_n(D\mid i,\bm{a})>0$.

Obviously, Assumption B implies Assumption A (by Proposition 1).

\begin{lemma} Under Assumption B, for any $\bm{\pi}=(\pi^k,k\in I)$ with $\pi^k=\{\phi_n^k,n\geq 0\}\in \Pi_m^k,   k\in I$, $n\geq 0$,
we define two sequences $\{u^m_{n,k}(\bm{\pi}),m\geq0\}$ and $\{v^m_{n,k}(\bm{\pi}^{-k}),m\geq0\}$ of functions $u^m_{n,k}(\bm{\pi})$ and $v^m_{n,k}(\bm{\pi}^{-k})$ on $ (S_n\setminus D)\times\mathbb{Q}^N$, respectively,  by
 \begin{eqnarray}
 u_{n,k}^{m+1}(i,\bm{\lambda},\bm{\pi})&:=& T_{n,k}^{\bm{\phi}_n}u_{n+1,k}^{m}(i,\bm{\lambda},\bm{\pi}), \ \ v_{n,k}^{m+1}(i,\bm{\lambda},\bm{\pi}^{-k}):=T_{n,k}^{\bm{\phi}_n^{-k}}v_{n+1,k}^{m}(i,\bm{\lambda},\bm{\pi}^{-k}) \label{E-1}\\
 {\rm with} \ u_{n,k}^{0}(i,\bm{\lambda},\bm{\pi})&=&v_{n,k}^{0}(i,\bm{\lambda},\bm{\pi}^{-k}):=\mathbb{I}_{(-\infty,0]}(\lambda^k)   \ \ \forall  (i,\bm{\lambda})\in (S_n\setminus D)\times \mathbb{Q}^N, n\geq 0.  \ \ \ \ \ \label{E-2}
   \end{eqnarray}
  Then, for all $n\geq 0, k\in I, m\geq 1,(i,\bm{\lambda})\in (S_n\setminus D)\times \mathbb{Q}^N$, we have
\begin{eqnarray*}
		|u_{n,k}^{m}(i,\bm{\lambda},\bm{\pi})-F_{n}^{k}(i,\bm{\lambda},\bm{\pi})| \leq \sum_{t=m}^{+\infty}(1-\beta)^{t}~~and~~	 |v_{n,k}^{m}(i,\bm{\lambda},\bm{\pi}^{-k})-\hat{F}_n^k(i,\bm{\lambda},\bm{\pi}^{-k})|\leq \sum_{t=m}^{+\infty}(1-\beta)^{t}.
	\end{eqnarray*}
\end{lemma}
\begin{proof} For any $m\geq 1$, by Theorem \ref{thm1}(c) and  the definition of  $u_{n,k}^{m+1}(\bm{\pi})$ above, we have
		\begin{eqnarray*}
		&&|u_{n,k}^{m+1}(i,\bm{\lambda},\bm{\pi})-F_n^k(i,\bm{\lambda},\bm{\pi})|\\
        &=&|{T}_{n,k}^{\bm{\phi}_n}u_{n+1,k}^m(i,\bm{\lambda},\bm{\pi})-T_{n,k}^{\bm{\phi}_n}F_{n+1}^k(i,\bm{\lambda},\bm{\pi})|\\
		&=&\bigg|\sum_{a^1\in A_n^1(i)}\ldots\sum_{a^N\in A_n^N(i)}\bigg[\sum_{j\in S_{n+1}\setminus D}\left( u_{n+1,k}^{m}(j,\bm{\lambda}-\bm{r}_n(i,\bm{a})),\bm{\pi})- F_{n+1}^k(j,\bm{\lambda}-\bm{r}_n(i,\bm{a})),\bm{\pi})\right)\\
		&&\times p_n(j\mid i,\bm{a})\bigg]\phi_{n,1}(a^1\mid i,\bm{\lambda})\ldots\phi_{n,N}(a^N\mid i,\bm{\lambda})\bigg|\\
		&\leq&(1-\beta)\| u_{n+1,k}^{m}(\bm{\pi})-F_{n+1}^{k}(\bm{\pi})\| \ \ \  \forall (i,\bm{\lambda})\in (S_n\setminus D)\times \mathbb{Q}^N,
	\end{eqnarray*}
	where the $\|h\|$ denotes the  sup norm of a bounded function $h$.  \\
Also,  by recursion we get
	\begin{eqnarray*}
	\|u_{n,k}^{m+1}(\bm{\pi})-u_{n,k}^m(\bm{\pi})\|
&\leq&(1-\beta)	\| u_{n+1,k}^{m}(\bm{\pi})- u_{n+1,k}^{m-1}(\bm{\pi})\| \\
&\leq&\cdots
		\leq(1-\beta)^m\|u_{n+l,k}^{1}(\bm{\pi})-u_{n+l,k}^{0}(\bm{\pi})\|\leq(1-\beta)^m.
	\end{eqnarray*}
	Moreover, we have
	\begin{eqnarray*}
		\|u_{n,k}^{m}(\bm{\pi})-F_{n}^{k}(\bm{\pi})\|
		&\leq& \|u_{n,k}^{m}(\bm{\pi})-u_{n,k}^{m+1}(\bm{\pi})\|+\|u_{n,k}^{m+1}(\bm{\pi})-F_{n}^{k}(\bm{\pi})\|\\
		&\leq& (1-\beta)^m+(1-\beta)\|u_{n+1,k}^{m}(\bm{\pi})-F_{n+1}^{k}(\bm{\pi})\|\leq\cdots\cdots\cdots\leq \sum_{t=m}^{+\infty}(1-\beta)^{t}.
	\end{eqnarray*}
	Due to $\hat{F}_n^k(i,\bm{\lambda},\bm{\pi}^{-k})=\sup_{\phi\in\mathscr{P}(A_n^k(i))}
T_{n,k}^{[\bm{\phi}_n^{-k},\phi]}\hat{F}_{n+1}^k(i,\bm{\lambda},\bm{\pi}^{-k})$ (by Theorem 2(a)) and\\
$v_{n,k}^{m+1}(i,\bm{\lambda},\bm{\pi}^{-k})=\sup_{\phi\in\mathscr{P}(A_n^k(i))}T_{n,k}^{[\bm{\phi}_n^{-k},\phi]}v_{n+1,k}^{m}(i,\bm{\lambda},\bm{\pi}^{-k})$ (by the definition of $v_{n,k}^{m+1}(\bm{\pi}^{-k})$),  we can similarly prove
$
	\|v_{n,k}^{m}(\bm{\pi}^{-k})-\hat{F}_n^{k}(\bm{\pi}^{-k})\|	\leq \sum_{t=m}^{+\infty}(1-\beta)^{t}$,
which gives the results.
\end{proof}

Next we give a definition of the $\epsilon$-Nash equilibrium as follows.
\begin{definition}
	For any given $\epsilon>0$,  $\bm{\pi}\in\bm{\Pi}$ is called an $\epsilon$-Nash equilibrium for the nonstationary nonzero-sum stochastic game,  if
	\begin{eqnarray*}
		F^k(i,\bm{\lambda},\bm{\pi})\geq\sup_{\pi^k\in\Pi^k}F^k(i,\bm{\lambda},\bm{\pi})-\epsilon~~\forall(i,\bm{\lambda})\in S_0\times \mathbb{Q}^N~ and~k\in I.
	\end{eqnarray*}
\end{definition}

To obtain  $\epsilon$-Nash equilibria, we provide {\color{black}the} algorithm below.
\begin{table}[H]
	\centering
	{\footnotesize
		\begin{tabular}{l}
			\toprule
			{\bf{Algorithm A.}} The algorithm to find an $\epsilon$-Nash equilibrium. \\
			\midrule
			{\bf{Step-I. Parameters}}  \\
			Accuracy parameter $\epsilon$, number of players $N$, period length $T(\epsilon):=[log_{(1-\beta)}^{(\frac{\epsilon}{5}\beta)}]^++1$, \\
			action set  $A_n^k$ for player $k\in I=\{1,\ldots,N\}$,  divide $[0,1]$ into $K=[\frac{10T(\epsilon)N\prod_{k\in I}(\max_{n\leq T(\epsilon)}|A_n^k|)}{\epsilon}]^++1$\\ {\color{black}equal segments}
			($|A_n^k|$ denotes the number of elements in $A_n^k$),
			$\delta:=\frac{1}{K}$,
			$L:=\{0, \delta, 2\delta,\ldots, K\delta\}$.\\
			{\bf{Step-II. Initialization}}  \\
			$\bullet$ For each  $m\ge0,$ $n\in\{0,1,\cdots,T(\epsilon)\}$ and $(i,\bm{\lambda})\in (S_n\setminus D)\times\mathbb{Q}^N$, \\choose the probability measures $\phi_{n,k}^{(m)}(\cdot\mid i,\bm{\lambda})=:(p_k^1,\ldots,p_k^{|A_n^k(i)|})$  such that   $p_k^t\in L$ and $\sum_{t=0}^{|A_n^k(i)|}p_k^t=1$.\\
			For each $n>T(\epsilon)$ and $(i,\bm{\lambda})\in (S_n\setminus D)\times\mathbb{Q}^N$, take $\phi_{n,k}^{(m)}(\cdot\mid i,\bm{\lambda})$ as any probability measures on $A_n^k(i)$.
			\\Construct a Markov policy $\pi^k_{(m)}=\{\phi_{n,k}^{(m)},n\geq 0\}$ for player $k$.\\ Let $\bm{\phi}_n^{(m)}:=(\phi_{n,1}^{(m)}, \phi_{n,2}^{(m)}, \ldots, \phi_{n,N}^{(m)})$ on $(S_n\setminus D)\times\mathbb{Q}^N$, and $\bm{\pi}_{(m)}:=\{\pi^1_{(m)},\pi^2_{(m)},\ldots,\pi^N_{(m)}\}$, set $m=0$.\\
			{\bf{Step-III. For $k=1,2,\ldots,N$ do.}}  \\
			{\bf{(a).}}\\
			{{\underline{Initialization}}}  \\
			$\bullet$ For each $(i,\bm{\lambda})\in (S_{T(\epsilon)}\setminus D)\times \mathbb{Q}^N$, let $u_{T(\epsilon),k}^{0}(i,\bm{\lambda},\bm{\pi}_{(m)})=\mathbb{I}_{(-\infty,0]}(\lambda^k)$, and take $l=1$.\\
			{\underline{Iteration}}  \\
			$\bullet$ Calculate the function $u_{T(\epsilon)-l,k}^{l}(i,\bm{\lambda},\bm{\pi}_{(m)})$ on $(S_{T(\epsilon)-l}\setminus D)\times \mathbb{Q}^N$, by \\
			\  \ \ \  \  \ \ \ \  \  \ \ \ \  \ \ \ \ \  \  \( u_{T(\epsilon)-l,k}^{l}(i,\bm{\lambda},\bm{\pi}_{(m)})=T_{T(\epsilon)-l,k}^{\bm{\phi}_{T(\epsilon)-l}^{(m)}}u_{T(\epsilon)-l+1,k}^{l-1}(i,\bm{\lambda},\bm{\pi}_{(m)}) \) \\
			{\underline{Stopping Rule}}  \\
			$\bullet$
			If $l=T(\epsilon)$,
			stop and we obtain $u_{0,k}^{T(\epsilon)}(i,\bm{\lambda},\bm{\pi}_{(m)})$.\\
		{\bf{(b).}}\\
		{{\underline{Initialization}}}  \\
		$\bullet$ For each $(i,\bm{\lambda})\in (S_{T(\epsilon)}\setminus D)\times \mathbb{Q}^N$, let $v_{T(\epsilon),k}^{0}(i,\bm{\lambda},\bm{\pi}_{(m)}^{-k})=\mathbb{I}_{(-\infty,0]}(\lambda^k)$, and take $l=1$.\\
		{\underline{Iteration}}  \\
		$\bullet$ Calculate the function $v_{T(\epsilon)-l,k}^{l}(i,\bm{\lambda},\bm{\pi}_{(m)}^{-k})$ on $(S_{T(\epsilon)-l}\setminus D)\times \mathbb{Q}^N$, by \\
		\  \ \ \  \  \ \ \ \  \  \ \ \ \  \ \ \ \ \  \  \( v_{T(\epsilon)-l,k}^{l}(i,\bm{\lambda},\bm{\pi}_{(m)}^{-k})=T_{T(\epsilon)-l,k}^{(\bm{\phi}_{T(\epsilon)-l}^{(m)})^{-k}}v_{T(\epsilon)-l+1,k}^{l-1}(i,\bm{\lambda},\bm{\pi}_{(m)}^{-k}) \) \\
		{\underline{Stopping Rule}}  \\
		$\bullet$
		If $l=T(\epsilon)$,
		stop and we obtain $v_{0,k}^{T(\epsilon)}(i,\bm{\lambda},\bm{\pi}_{(m)}^{-k})$.\\
			{\bf{(c).}}\\
			{\underline{Comparison}}  \\
			$\bullet$ For each $(i,\bm{\lambda})\in (S_0\setminus D)\times \mathbb{Q}^N$. If\\
			\ \ \  \  \ \ \ \  \  \ \ \ \  \ \ \ \ \  \  \(
			|u_{0,k}^{T(\epsilon)}(i,\bm{\lambda},\bm{\pi}_{(m)})- v_{0,k}^{T(\epsilon)}(i,\bm{\lambda},\bm{\pi}_{(m)}^{-k})|<\frac{3\epsilon}{5}\)~~~~~~({\bf{c-1*}})\\
			then go on, otherwise, return to {\bf{(B)} }, let $m=m+1$ and reselect the  $\bm{\pi}_{(m)}$.\\
			{\bf{end for.}} \\
			{\bf{Step-IV. Stopping Rule}}  \\
			$\bullet$  If for each $k\in\{1,2,\ldots,N\}$, {\bf{(c-1*)}} holds, take $G(\epsilon):=m$, then $\bm{\pi}_{(G(\epsilon))}$  is an $\epsilon$-Nash equilibrium.\\
			\bottomrule
		\end{tabular}
	}
\end{table}

In Theorem 3, we prove that  {\color{black}the} $\epsilon$-Nash equilibrium  can be efficiently found by using Algorithm 1.
\begin{theorem}\label{thm3} Under Assumption B, given any $\epsilon>0$, there exists a finite $G(\epsilon):=\prod_{n=0}^{T(\epsilon)}(K+1)^{\prod_{k\in I}|A_n^k(i)|}$ such that the policy $\bm{\pi}_{(G(\epsilon))}$ from Algorithm 1  is an $\epsilon$-Nash equilibrium.
\end{theorem}
\begin{proof} (1) We first prove that if  an $\bm{\pi}_{(G(\epsilon))}=\{\pi^1_{(G(\epsilon))},\pi^2_{(G(\epsilon))},\ldots,\pi^N_{(G(\epsilon))}\}$  satisfies {\bf{(c-1*)}}, then $\bm{\pi}_{(G(\epsilon))}$ is an $\epsilon$-Nash equilibrium .
	
	By Lemma 2 and $T(\epsilon):=[log_{(1-\beta)}^{(\frac{\epsilon}{5}\beta)}]^++1$,  a simple calculation gives
		\begin{eqnarray}	&&|F_0^k(i,\bm{\lambda},\bm{\pi}_{(G(\epsilon))})-u_{0,k}^{T(\epsilon)}(i,\bm{\lambda},\bm{\pi}_{(G(\epsilon))})|\leq \sum_{t=T(\epsilon)}^{+\infty}(1-\beta)^{t}<\frac{\epsilon}{5},\label{C-1}\\
		&&	|\hat{F}_0^k(i,\bm{\lambda},\bm{\pi}^{-k}_{(G(\epsilon))})-v_{0,k}^{T(\epsilon)}(i,\bm{\lambda},\bm{\pi}_{(G(\epsilon))}^{-k})|\leq \sum_{t=T(\epsilon)}^{+\infty}(1-\beta)^{t}<\frac{\epsilon}{5}.\label{C-2}
	\end{eqnarray}
Thus, using {\bf{(c-1*)}} we have
\begin{eqnarray*}
	&&|F_0^k(i,\bm{\lambda},\bm{\pi}_{(G(\epsilon))})-\hat{F}_0^k(i,\bm{\lambda},\bm{\pi}^{-k}_{(G(\epsilon))})|\\
	&\leq&
	|F_0^k(i,\bm{\lambda},\bm{\pi}_{(G(\epsilon))})-u_{0,k}^{T(\epsilon)}(i,\bm{\lambda},\bm{\pi}_{(G(\epsilon))})|
	+	|u_{0,k}^{T(\epsilon)}(i,\bm{\lambda},\bm{\pi}_{(G(\epsilon))})- v_{0,k}^{T(\epsilon)}(i,\bm{\lambda},\bm{\pi}_{(G(\epsilon))}^{-k})|\\
	&&+|\hat{F}_0^k(i,\bm{\lambda},\bm{\pi}^{-k}_{(G(\epsilon))})-v_{0,k}^{T(\epsilon)}(i,\bm{\lambda},\bm{\pi}_{(G(\epsilon))}^{-k})|\\
	&\leq&\epsilon,
	\end{eqnarray*}
which means that $\bm{\pi}_{(G(\epsilon))}$ is an $\epsilon$-Nash equilibrium.

	(2) Next, we will prove that there exists a finite $G(\epsilon)$ satisfying {\bf{(c-1*)}}, i.e.,
	\begin{eqnarray}\label{c-2}
		|u_{0,k}^{T(\epsilon)}(i,\bm{\lambda},\bm{\pi}_{(G(\epsilon))})- v_{0,k}^{T(\epsilon)}(i,\bm{\lambda},\bm{\pi}_{(G(\epsilon))})|<\frac{3\epsilon}{5}.
	\end{eqnarray}
Take  the Nash equilibrium  $\bm{\pi}_*=(\pi^{1}_*,\pi^{2}_*,\cdots,\pi^{N}_*)$ from Theorem \ref{thm2},  where $\pi^{k}_*:=\{\psi^k_{n,*}, n\ge0\}$, let $\bm{\psi}_n^*:=(\psi^k_{n,*})_{k\in I}$.  Given any $\bm{\pi}=(\pi^k,k\in I)$ with $\pi^k=\{\phi_{n,k},n\ge0\}$, by induction we next show the following fact: If
\begin{eqnarray}\label{W-1}
\max_{k\in I}\rho^k(\psi^k_{n,*}(\cdot|i,\bm{\lambda}), \phi_{n,k}(\cdot|i,\bm{\lambda}))<\delta \ \ \  \forall n\in\{0,1,\cdots,T(\epsilon)\}, (i,\bm{\lambda})\in (S_n\setminus D)\times \mathbb{Q}^N,
\end{eqnarray}
 where  $\rho^k(\phi, \phi'):=\max_{a\in A_n^k(i)}|\phi(a)-\phi'(a)|$, then, for each $k\in I, (i,\bm{\lambda})\in (S_n\setminus D)\times \mathbb{Q}^N$,
\begin{eqnarray}\label{A3}
	|u_{T(\epsilon)-l,k}^{l}(i,\bm{\lambda},\bm{\pi}_*)-u_{T(\epsilon)-l,k}^{l}(i,\bm{\lambda},\bm{\pi})|\leq lN\delta \prod_{k\in I}(\max_{n\leq T(\epsilon)}|A_n^k|)~~\forall~0\leq l\leq T(\epsilon).
\end{eqnarray}
Indeed, (\ref{A3})  is true for $l=0$ (by (\ref{E-2})). We now assume that it holds for some $0\leq l\leq T(\epsilon)$.
Then, for $l+1\leq T(\epsilon)$, let $T(\epsilon)-l-1=:n$. By  the induction hypothesis and (\ref{E-1}), we have
\begin{eqnarray*}
	&&|u_{n,k}^{l+1}(i,\bm{\lambda},\bm{\pi}_*)-u_{n,k}^{l+1}(i,\bm{\lambda},\bm{\pi})|\\
	&=&	|T_{n,k}^{\bm{\psi}_n^*}u_{n+1,k}^{l}(i,\bm{\lambda},\bm{\pi}_*)-T_{n,k}^{\bm{\phi}_n}u_{n+1,k}^{l}(i,\bm{\lambda},\bm{\pi})|\\
	&=&
	\bigg|\sum_{a^1\in A_n^1(i)}\ldots\sum_{a^N\in A_n^N(i)}\bigg[\mathbb{I}_{[\lambda^k,+\infty)}(r_n^k(i,\bm{a}))
	p_n(D\mid i,\bm{a})\nonumber\\
	&&~~~+\sum_{j\in S_{n+1}\setminus D}u_{n+1,k}^{l}(j,\bm{\lambda}-\bm{r}_n(i,\bm{a}),\bm{\pi}_*)p_n(j\mid i,\bm{a})\bigg]\psi_{n,*}^1(a^1\mid i,\bm{\lambda} )\ldots\psi_{n,*}^N(a^N\mid i,\bm{\lambda})\\
	&&-\sum_{a^1\in A_n^1(i)}\ldots\sum_{a^N\in A_n^N(i)}\bigg[\mathbb{I}_{[\lambda^k,+\infty)}(r_n^k(i,\bm{a}))
	p_n(D\mid i,\bm{a})\nonumber\\
	&&~~~+\sum_{j\in S_{n+1}\setminus D}u_{n+1,k}^{l}(j,\bm{\lambda}-\bm{r}_n(i,\bm{a}),\bm{\pi})p_n(j\mid i,\bm{a})\bigg]\phi_{n,1}(a^1\mid i,\bm{\lambda} )\ldots\phi_{n,N}(a^N\mid i,\bm{\lambda}) \bigg|\\
    &=&\bigg|	\sum_{a^1\in A_n^1(i)}\ldots\sum_{a^N\in A_n^N(i)}\\
	&&\bigg[\sum_{j\in S_{n+1}\setminus D}[u_{n+1,k}^{l}(j,\bm{\lambda}-\bm{r}_n(i,\bm{a}),\bm{\pi}_*)-u_{n+1,k}^{l}(j,\bm{\lambda}-\bm{r}_n(i,\bm{a}),\bm{\pi})]p_n(j\mid i,\bm{a})\bigg]\\
	&&\times\psi_{n,*}^1(a^1\mid i,\bm{\lambda} )\ldots\psi_{n,*}^N(a^N\mid i,\bm{\lambda})\\
	&&+ \bigg|\sum_{a^1\in A_n^1(i)}\ldots\sum_{a^N\in A_n^N(i)}\\
	&&\bigg[\mathbb{I}_{[\lambda^k,+\infty)}(r_n^k(i,\bm{a}))
	p_n(D\mid i,\bm{a})+\sum_{j\in S_{n+1}\setminus D}u_{n+1,k}^{l}(j,\bm{\lambda}-\bm{r}_n(i,\bm{a}),\bm{\pi})p_n(j\mid i,\bm{a})\bigg]\\
	&&\times(\psi_{n,*}^1(a^1\mid i,\bm{\lambda} )\ldots\psi_{n,*}^N(a^N\mid i,\bm{\lambda})-
	\phi_{n,1}(a^1\mid i,\bm{\lambda} )\ldots\phi_{n,N}(a^N\mid i,\bm{\lambda}))  \bigg|\\
	&\leq&lN\delta \prod_{k\in I}(\max_{n\leq T(\epsilon)}|A_n^k|)\\
	&&+\left|\sum_{a^1\in A_n^1(i)}\ldots\sum_{a^N\in A_n^N(i)}(\psi_{n,*}^1(a^1\mid i,\bm{\lambda} )\ldots\psi_{n,*}^N(a^N\mid i,\bm{\lambda})-
	\phi_{n,1}(a^1\mid i,\bm{\lambda} )\ldots\phi_{n,N}(a^N\mid i,\bm{\lambda}))\right|\\
	&\leq& lN\delta \prod_{k\in I}(\max_{n\leq T(\epsilon)}|A_n^k|)+N\delta\prod_{k\in I} (\max_{n\leq T(\epsilon)}|A_n^k|)=(l+1)N\delta \prod_{k\in I}(\max_{n\leq T(\epsilon)}|A_n^k|),
\end{eqnarray*}
where the first part of the first inequality  is by induction hypothesis, and  the second inequality  is due to the following  fact:
\begin{eqnarray*}
	&&\sum_{a^1\in A_n^1(i)}\ldots\sum_{a^N\in A_n^N(i)}(\psi_{n,*}^1(a^1\mid i,\bm{\lambda} )\ldots\psi_{n,*}^N(a^N\mid i,\bm{\lambda})-
\phi_{n,1}(a^1\mid i,\bm{\lambda} )\ldots\phi_{n,N}(a^N\mid i,\bm{\lambda}))\\
&\leq& N\delta\prod_{k\in I} (\max_{n\leq T(\epsilon)}|A_n^k|),
\end{eqnarray*}
which is proved as follows:
\begin{eqnarray*}
	&&\left|\sum_{a^1\in A_n^1(i)}\ldots\sum_{a^N\in A_n^N(i)}(\psi_{n,*}^1(a^1\mid i,\bm{\lambda} )\ldots\psi_{n,*}^N(a^N|i,\bm{\lambda})-
	\phi_{n,1}(a^1|i,\bm{\lambda} )\ldots\phi_{n,N}(a^N|i,\bm{\lambda}))\right|\\
	&\leq& \prod_{k\in I}(\max_{n\leq T(\epsilon)}|A_n^k|) \max_{\bf{a}\in A_n}\big|\psi_{n,*}^1(a^1|i,\bm{\lambda} )\ldots\psi_{n,*}^N(a^N|i,\bm{\lambda})-
	\phi_{n,1}(a^1|i,\bm{\lambda} )\ldots\phi_{n,N}(a^N|i,\bm{\lambda})\big|\\
	&=&\prod_{k\in I}(\max_{n\leq T(\epsilon)}|A_n^k|) \max_{\bf{a}\in A_n}\\
	&& \{\big|\psi_{n,*}^1(a^1|i,\bm{\lambda} )\psi_{n,*}^2(a^2|i,\bm{\lambda} )\ldots\psi_{n,*}^N(a^N|i,\bm{\lambda})-\phi_{n,1}(a^1|i,\bm{\lambda} )\psi_{n,*}^2(a^2| i,\bm{\lambda} )\ldots\psi_{n,*}^N(a^N|i,\bm{\lambda})\\
	&&+ \phi_{n,1}(a^1|i,\bm{\lambda} )\psi_{n,*}^2(a^2|i,\bm{\lambda} )\ldots\psi_{n,*}^N(a^N|i,\bm{\lambda})-\phi_{n,1}(a^1|i,\bm{\lambda} )\phi_{n,2}(a^2|i,\bm{\lambda} )\ldots\phi_{n,N}(a^N|i,\bm{\lambda})\big|\}\\
	&\leq&\prod_{k\in I}(\max_{n\leq T(\epsilon)}|A_n^k|) \max_{\bf{a}\in A_n}\{\big|\psi_{n,*}^1(a^1|i,\bm{\lambda} )-\phi_{n,1}(a^1|i,\bm{\lambda} )\big|\times\big|\psi_{n,*}^2(a^2|i,\bm{\lambda} )\ldots\psi_{n,*}^N(a^N|i,\bm{\lambda})\big|\\
	&&+ \big|\phi_{n,1}(a^1|i,\bm{\lambda} )\big|\times\big|\psi_{n,*}^2(a^2|i,\bm{\lambda} )\ldots\psi_{n,*}^N(a^N|i,\bm{\lambda})-\phi_{n,2}(a^2|i,\bm{\lambda} )\ldots\phi_{n,N}(a^N|i,\bm{\lambda} )\big|\}\\
	&\leq&\prod_{k\in I}(\max_{n\leq T(\epsilon)}|A_n^k|)(\delta+ \max_{\bm{a}\in \bm{A}_n}\big|\psi_{n,*}^2(a^2|i,\bm{\lambda} )\ldots\psi_{n,*}^N(a^N|i,\bm{\lambda}) )-\phi_{n,2}(a^2|i,\bm{\lambda} )\ldots\phi_{n,N}(a^N|i,\bm{\lambda} )\big|)\\
	&\leq &  \prod_{k\in I}(\max_{n\leq T(\epsilon)}|A_n^k|)N\delta.
\end{eqnarray*}
Thus we have proved (\ref{A3}), and (by  the parameters in Algorithm 1) obtain that
\begin{eqnarray}\label{A4}
	|u_{0,k}^{T(\epsilon)}(i,\bm{\lambda},\bm{\pi}_*)-u_{0,k}^{T(\epsilon)}(i,\bm{\lambda},\bm{\pi})|
	\leq T(\epsilon)N\delta \prod_{k\in I}(\max_{n\leq T(\epsilon)}|A_n^k|)<\frac{\epsilon}{10}.
\end{eqnarray}
Following the arguments of (\ref{A4}), by (\ref{E-1}) we  can also prove that
\begin{eqnarray}\label{A5}
	|v_{0,k}^{T(\epsilon)}(i,\bm{\lambda},\bm{\pi}_*^{-k})-v_{0,k}^{T(\epsilon)}(i,\bm{\lambda},\bm{\pi}^{-k})|<\frac{\epsilon}{10}~~\forall~(i,\bm{\lambda})\in (S_n\setminus D)\times \mathbb{Q}^N,k\in I.
\end{eqnarray}
Then it can be gained that
\begin{eqnarray*}
	u_{0,k}^{T(\epsilon)}(i,\bm{\lambda},\bm{\pi})&>&u_{0,k}^{T(\epsilon)}(i,\bm{\lambda},\bm{\pi}_*)-\frac{\epsilon}{10}~~~~~~~~~~~~~~~~~~~~~~(by ~(\ref{A4}))\\
	&>&F_0^k(i,\bm{\lambda},\bm{\pi}_*)-\frac{\epsilon}{5}-\frac{\epsilon}{10}~~~~~~~~~~~~~~~~~~~(by ~(\ref{C-1}))\\
	&=&\hat{F}_0^k(i,\bm{\lambda},\bm{\pi}_*^{-k})-\frac{3\epsilon}{10}~~~~~~~~~~~~~~~~~~(by~ Theorem~2)\\
	&>&v_{0,k}^{T(\epsilon)}(i,\bm{\lambda},\bm{\pi}_*^{-k})-\frac{3\epsilon}{10}-\frac{\epsilon}{5}~~~~~~~~~~~~~~~(by ~(\ref{C-2}))\\
	&>&v_{0,k}^{T(\epsilon)}(i,\bm{\lambda},\bm{\pi}^{-k})-\frac{3\epsilon}{5}.~~~~~~~~~~~~~~~~~~~~~(by ~(\ref{A5}))
\end{eqnarray*}
Thus, for obtaining (\ref{c-2}), we  need to prove the existence of a finite integer $G(\epsilon)$ such that {\color{black}(\ref{W-1}) holds}.
For the Nash equilibrium $\bm{\pi}_*=\{\pi^{k}_*\}_{k\in I}$ with $\pi^{k}_*=\{\psi^k_{n,*}, n\ge0\}$, for any $k\in I, 0\leq n\leq T(\epsilon), (i,\bm{\lambda})\in (S_n\setminus D)\times\mathbb{Q}^N$,   let $\psi^k_{n,*}(\cdot| i,\bm{\lambda})=:(p^{*1},p^{*2},\ldots,p^{*|A_n^k(i)|})$.
Then for each $p^{*t} (\in[0,1])$, we can choose an $ p^t\in L=\{0, \delta, 2\delta,\ldots, K\delta\}$ such that  $|p^{*t}-p^t|<\delta ~\forall ~t\in\{1,2,\ldots,|A_n^k(i)|\}$ and $\sum_{t=1}^{|A_n^k(i)|}p^t=1$. Let $\phi_{n,k}^{(G(\epsilon))}(\cdot|i,\bm{\lambda}):=(p^1,p^2,\ldots,p^{|A_n^k(i)|})$. Then,  we have $\max_{k\in I}\rho^k(\psi^k_{n,*}(\cdot|i,\bm{\lambda}), \phi_{n,k}^{(G(\epsilon))}(\cdot|i,\bm{\lambda}))=\max_{k\in I}\max_{t\in \{1,...,|A_n^k(i)|\}}|p^{*t}-p^t|<\delta$.
Since $p^t\in \{0, \delta, 2\delta,\ldots, K\delta\}$, the number of {\color{black}choices of} ``$p^t$" is at most $K+1$. Thus, the number of constructing $\phi_{n,k}^{(G(\epsilon))}(\cdot| i,\bm{\lambda})$ is at most $(K+1)^{|A_n^k(i)|}$. Therefore, the {\color{black}number}  of {\color{black}choices of}  $\bm{\phi}_n^{(G(\epsilon))}(\cdot|i,\bm{\lambda}):=(\phi_{n,1}^{(G(\epsilon))}(\cdot|i,\bm{\lambda}), \phi_{n,2}^{(G(\epsilon))}(\cdot| i,\bm{\lambda}), \ldots, \phi_{n,N}^{(G(\epsilon))}(\cdot|i,\bm{\lambda}))$ is  at most $(K+1)^{\prod_{k\in I}|A_n^k(i)|}$. This means that the number $m$ of {\color{black} selections} $\pi(m)$ in Algorithm 1 is finite, that is, $G(\epsilon)=\prod_{n=0}^{T(\epsilon)}(K+1)^{\prod_{k\in I}|A_n^k(i)|}$ is finite.

\end{proof}

\section{Examples}\label{sec5}
In this section we illustrate the applications of our main results with an energy management model and take a numerical experiment.

 \begin{example}(A nonstationary energy management system): \label{ex1}
 	\\
 	In the considered system, without loss of generality, suppose that there are three players. One of them,  called Player 1, who
 is equipped with a wind turbine,  a solar panel and the concentrating solar power(CSP), owns 
 	{\color{black}some renewable energy sources}, and can both consume energy as well as provide energy to other two customers (say, Player 2 and Player 3).
 	
 	To build the model of  the system, we use notation $i_n^k,M_n^k, a_n^k$, $\xi_n^k$, $\eta····························_n^k$, where
 \begin{itemize}	
 \item[$M_n^k$:] denotes the maximum storage capacity of Player $k=1,2,3$ during period $n\ge0$;
 \item[$i_n^k$:] denotes the energy storage level for Player $k$ at the beginning of period $n\ge0$;
 \item[$a_n^{1,k}$:] is the amount of energy that Player 1 may sell to Player $k=2,3$ during  period $n\ge0$;
 \item[$a_n^k$:] is the  energy demanded by Player $k$ from Player 1 during  period $n\ge0$, for $k=2,3$;
 \item[$z_n^k$:] denotes the energy consumed minus the energy harvested by the wine turbine, the solar panel and the CSP, respectively, for $k=1,2,3$;
\item[$\xi_n^1$:]  denotes the actual consumed/generated energy of Player 1 during period $n$ satisfying $\xi_n^1=z_n^1+z_n^2+z_n^3$, which is the energy consumed minus the energy harvested (from renewable resources) and assumed to have a distribution $P(\xi_n^1=m)=:q_n^1(m)$ with $m=0,\pm 1,\ldots,\pm M_n^1$;
\item[$\xi_n^k$:] denotes the energy consumed of Player $k$ during period $n$ and  is assumed to have a distribution  $P(\xi_n^k=m)=:q_n^k(m)$ with $m=0,1,2,\ldots$ for $k=2,3$;
\item[$\eta_n^k$:]  denotes the additional energy bought from the utility company (rather than player 1) by Player $k$ during period $n$ and is assumed to have a distribution  $P(\eta_n^k=m)=:g_n^k(m)$ with $m=0,\ldots,M_n^k$ for $k=2,3$.
 \end{itemize}
 	 Since the actual trading energy between Player 1 and Player $k\in\{2,3\}$ at period $n$ is $\min\{a_n^{1,k},a_n^k\}$, using the standard notation $x^+:=\max\{0,x\}$, we can obtain that the storage level of Player $k$ evolves according to the following equations
 	 \begin{eqnarray}
 	 	i_{n+1}^1&:=&\min\{M_{n+1}^1,(i_n^1-\min\{a_n^{1,2},a_n^2\}-\min\{a_n^{1,3},a_n^3\}-\xi_n^1)^+\},\\
 	 	i_{n+1}^k&:=&\min\{M_{n+1}^k,(i_n^k+\min\{a_n^{1,k},a_n^k\}-\xi_n^k+\eta_n^k)^+\}, \ \  k=2,3,n\ge0,
 	 \end{eqnarray}
 given initial stock levels $i_0^k(k=1,2,3)$.
 	
 	 Thus, for each $n\ge0$, the state space  during period $n$ is $S_n:=\prod\limits_{k=1}^3\{0,\ldots,M_n^k\}$. For each state $(i_1,i_2,i_3)\in S_n$, from the statements of the system above,the available action sets may be taken as follows: $A_n^1(i_1,i_2,i_3)\subseteq\{(b_1,b_2):b_1\in \{0,\ldots,i_1\},b_2\in\{0,\ldots,i_1-b_1\}\}$, $A_n^2(i_1,i_2,i_3)\subseteq\{0,\ldots,M_n^2-i_2\}, A_n^3(i_1,i_2,i_3)\subseteq\{0,\ldots,M_n^3-i_3\}$.

  Suppose  that $\{\xi_n^k\}(k=1,2,3)$ and $\{\eta_n^k\}(k=2,3)$ are independent. Then, for $(i_1,i_2,i_3)\in S_n$,  $a^1:=(a^{1,2},a^{1,3})\in A^1_n(i_1,i_2,i_3), a^2\in A^2_n(i_1,i_2,i_3),a^3\in A^3_n(i_1,i_2,i_3),$  and $(j_1,j_2,j_3)\in S_{n+1}$, the transition laws $p_n$ are given by
 	 \begin{eqnarray*}
 	 	&&p_n((j_1,j_2,j_3)|(i_1,i_2,i_3),(a^{1},a^{2},a^3))\\
 	 	&=&\sum_{m=-M_{n}^1}^{M_n^1}\mathbb{I}_{\{j_1\}}(\min\{M_{n+1}^1,(i_1-\min\{a^{1,2},a^2\}-\min\{a^{1,3},a^3\}-m)^+)q_n^1(m)\\
 	 	&&\times \sum_{m=-M_n^2}^{+\infty} \mathbb{I}_{\{j_2\}}(\min\{M_{n+1}^2,(i_2+\min\{a^{1,2},a^2\}-m)^+\})\hat q_n^{2}(m)\\
 	 	&&\times \sum_{m=-M_n^3}^{+\infty} \mathbb{I}_{\{j_3\}}(\min\{M_{n+1}^3,(i_3+\min\{a^{1,3},a^3\}-m)^+\})\hat q_n^{3}(m),
 	 	\end{eqnarray*}
  	where for $k=2,3$,
  	\begin{equation}
  		\hat q_n^{k}(m):=P(\xi_n^k-\eta_n^k=m)=
  		\left\{
  		\begin{array}{lr}
  		\sum_{l=0}^{M_n^k}q_n^k(m+l)g_n^k(l),m=0,1,2\ldots\\
  		\sum_{l=|m|}^{M_n^k}q_n^k(l+m)g_n^k(l),m=-1,\ldots,-M_n^k.
  			\end{array}
  		\right.
  	\end{equation}
  Finally, to complete the description of the control model (1) for this energy management system, we shall consider  reward functions $r_n^k$ on the corresponding $\mathbb{K}_n$ during period $n$ for Players $k$, which are assumed to be nonnegative-valued.
 	
 	 We aim to ensure the existence of a Nash equilibrium with the maximum security probability before each player' energy in stock becomes 0. To achieve the goal, in spirit of Assumption A, we consider the following condition:

Condition C: $\sum_{n=0}^\infty \beta_n=\infty$, where $\beta_n:=q_n^1(M_n^1)\sum_{m\ge M_n^2}^{+\infty}\hat q_n^{2}(m)\sum_{m\ge M_n^3}^{+\infty}\hat q_n^{3}(m)$.
\end{example}
Condition C can ensure the existence of a Nash equilibrium. Before proving this assertion, we provide some sufficient conditions for Condition C below.
\begin{proposition}
	For Example 1 with finite capacities $M_*^k$ ($i.e., M_n^k\leq M_*^k<\infty$ for all $n\ge 0$), one of the following hypotheses (a)-(b) implies  Condition C.
	\\(a) The case of producing and selling of Player 1 priority to trades to Players 2 and 3:
	\begin{itemize}
		\item[($a_1$)]$\sum_{n\ge0}q_n^1(M_n^1)=\infty.$ \ \ \ \ (Specially, if there exists a constant $\delta\in(0,1)$ such that $q_n^1(M_n^1)\ge \frac{\delta}{n+M_n^1}$ for all $n\ge 0$, the the hypothesis ($a_1$) is satisfied.)
		\item[($a_2$)] $\inf_{n\ge0}\sum_{m\ge M_n^k}^{+\infty}\hat q_n^{k}(m)>0$ for $k=2,3$. \ \ \ \ (Specially, if $\{\xi_n^k\}$ and $\{\eta_n^k\}$ have common distributions $q^k(\cdot)$ and $g^k(\cdot)$, respectively, and $\inf_{n\ge0}\sum_{m\ge M_*^k}^{+\infty}\sum_{l=0}^{M_n^k}q^k(m+l)g^k(l)>0$, then the hypothesis ($a_2$) is satisfied.)
	\end{itemize}
	(b) The case of ordering and selling of Players 2 and 3 priority to trades to Player 1:
	\begin{itemize}
		\item[($b_1$)]$\inf_{n\ge0}q_n^1(M_n^1)>0.$ \ \ \ \  (Specially, if $\{\xi_n^1\}$ have a common distribution $q^1(\cdot)$ and $q^1(M_n^1)>0$, the the hypothesis ($b_1$) is satisfied.)
		\item[($b_2$)] $\sum_{n\ge0}\sum_{m\ge M_n^k}^{+\infty}\hat q_n^{k}(m)=\infty$ for $k=2,3.$ \ \ \ \  (Specially, if there exists a constant $\delta_2^k\in(0,1)$ such that $q_n^k(m)\ge \frac{\delta_2^k}{(n+m)(n+m+1)}$ for all $n,m\ge0$ and $\inf_{n\ge0,o\leq m\leq M_n^k}g_n^k(m):=\delta_3^k>0$, then the hypothesis ($b_2$) is satisfied.)
	\end{itemize}
\end{proposition}
\begin{proof}
	(a) Obviously, it suffices to verify the special case. In fact, for the special case, we have $\sum_{n\ge0}q_n^1(M_n^1)\ge \sum_{n\ge0}\frac{\delta}{n+M_n^1}\ge  \sum_{n\ge0}\frac{\delta}{n+M_*^1}$ and $\delta_k:=\inf_{n\ge0}\sum_{m\ge M_k}^{+\infty}\sum_{l=0}^{M_n^k}q^k(m+l)g^k(l)>0$ for $k=2,3$ and $n\ge0$, which gives $\sum_{n=0}^{+\infty}\beta_n\ge\sum_{n=0}^{+\infty}\frac{\delta\delta_2\delta_3}{n+M_*^1}=\infty$.
	\\(b) It suffices to verify (b2) for the special case. Indeed, for the  case we have for $k=2,3$,
	\begin{eqnarray*}
		\sum_{n\ge0}\sum_{m\ge M_n^k}^{+\infty}\hat q_n^{k}(m)&=&\sum_{n\ge0}\sum_{m\ge M_n^k}^{+\infty}\sum_{l=0}^{M_n^k}q_n^k(m+l)g_n^k(l)\\
		&\ge&\sum_{n\ge0}\sum_{m\ge M_n^k}^{+\infty}\sum_{l=0}^{M_n^k}\frac{\delta_2^k\delta_3^k}{(n+m+k)(n+m+k+1)}\\
		&\ge&\sum_{n\ge0}\sum_{m\ge M_n^k}^{+\infty}\frac{\delta_2^k\delta_3^k}{(n+m)(n+m+1)}=\sum_{n\ge0}\frac{\delta_2^k\delta_3^k}{n+M_n^k}=\infty.
	\end{eqnarray*}
	Hence, (b2) is true, and thus Condition C is satisfied.
\end{proof}

\begin{proposition}
Under the Condition C, there exists a Nash equilibrium for the nonstationary energy management system above.
\end{proposition}
\begin{proof} In fact, let $D=:\{0\}\times\{0\}\times\{0\}$ for each $n\ge 0$, $(i_1,i_2,i_3)\in (S_n\setminus D)$ and $a^1:=(a^{1,2},a^{1,3})\in A^1_n(i_1,i_2,i_3),
a^2\in A^2_n(i_1,i_2,i_3),a^3\in A^3_n(i_1,i_2,i_3),$ we have
\begin{eqnarray*}
&&p_n(\{0\}\times\{0\}\times\{0\}|a^1,a^2,a^3)\\
&=&\sum_{m=-M_n^1}^{M_n^1}\mathbb{I}_{\{0\}}(\min\{M_{n+1}^1,(i_1-\min\{a^{1,2},a^{2}\}-\min\{a^{1,3},a^3\}-m)^+)q_n^1(m)\\
&&\times \sum_{m=-M_n^2}^{+\infty} \mathbb{I}_{\{0\}}(\min\{M_{n+1}^2,(i_2+\min\{a^{1,2},a^2\}-m)^+\})\hat q_n^{2}(m)\\
&&\times \sum_{m=-M_n^3}^{+\infty} \mathbb{I}_{\{0\}}(\min\{M_{n+1}^3,(i_3+\min\{a^{1,3},a^3\}-m)^+\})\hat q_n^{3}(m)\\
&=&\sum_{m=i_1-\min\{a^{1,2},a^{2}\}-\min\{a^{1,3},a^3\}}^{M_n^1}q_n^1(m)
\times \sum_{m=i_2+\min\{a^{1,2},a^{2}\}}^{+\infty}\hat q_n^{2}(m)
\times \sum_{m=i_3+\min\{a^{1,3},a^3\}}^{+\infty}\hat q_n^{3}(m)\\\
&\ge&q_n^1(M_n^1)\sum_{m\ge M_n^2}^{+\infty}\hat q_n^{2}(k)\sum_{m\ge M_n^3}^{+\infty}\hat q_n^{3}(k)=\beta_n,
\end{eqnarray*}
which, together with the Condition C and  Proposition 1, implies that Assumption A holds. Hence, the conclusion follows from  Theorem 2.
\end{proof}
\begin{remark}\label{rem3}
For Example 1 with a finite capacity $M^*$, suppose that $q_n(k)$ satisfies $q_n(k)\geq \frac{\delta
}{(n+k)(n+k+1)}$ for all $k\ge M^*$ and $n\geq 0$, with some constant  $\delta\in(0,1)$. Then, we have  $\beta_n\ge\sum_{k=M^*}^{\infty} q_n(k)\geq \frac{\delta}{n+M^*}$ for all $n\geq 0$. Thus, the Condition C (hence, Assumption A) is satisfied. However, if  taking $q_n(k)=\frac{\delta
}{(n+k)(n+k+1)}$ for all $k\ge M^*$ and $n\geq 0$, then we let $\beta_n:=\sum_{k=M^*}^{\infty}  q_n(k)=\frac{\delta}{n+M^*}$ and so $\inf_{n\geq 0} \beta_n=0$, which implies that both the Condition 1 in \cite{Huang2020} fail to hold. Moreover, if $\sum_{k=M^*}^{\infty} q_n(k)\geq \beta$ for all $n\geq 0$, with some positive constant $\beta$, then Assumption A is satisfied.
\end{remark}
{\color{black} \begin{example}(A numerical experiment) Consider the safety problem of two insurance companies (we say ``A" and ``B") in two states of boom ``1'' and slump  ``2" in financial markets. Both companies can choose to invest or not invest in the boom state ($``a_{11}\& a_{12}"$ for company A, $``b_{11}\& b_{12}"$ for company B,), and choose to exit the market in the slump state ($``a_{21}"$ for A, $``b_{21}"$ for  B). Suppose that the actions of both affect their respective rewards. Both {companies} want to maximize the probability that accumulative rewards will exceed their profit goal before financial markets enter a slump.
 	 For the numerical calculation, we take {some specific data}: $S_n\equiv \{1,2\}$, $D=\{2\}$, $A_n^1(1)\equiv \{a_{11},a_{12}\}$,
 	$A_n^1(2)\equiv \{a_{21}\}$,
 	 $A_n^2(1)\equiv \{b_{11},b_{12}\}$
 	 $A_n^2(2)\equiv \{b_{21}\}$,
 	 $p_n(1\mid 1, a_{11},b_{11})\equiv 0.55$,
 	  $p_n(1\mid 1, a_{11},b_{12})\equiv 0.60$,
 	   $p_n(1\mid 1, a_{12},b_{11})\equiv 0.45$,
 	    $p_n(1\mid 1, a_{12},b_{12})\equiv 0.60$ for all $n\ge 0$.
 	   Moreover, we suppose that \\
 	   $
 	   r_0^1(1,a,b)=\left\{
 	   \begin{aligned}
 	   &1&& if \quad (a,b)=(a_{11},b_{11}),\\	   &0&& if \quad (a,b)=(a_{11},b_{12}),\\
 	   &0&& if \quad (a,b)=(a_{12},b_{11}),\\	   &1&& if \quad (a,b)=(a_{12},b_{12}),\\
 	   \end{aligned}
 	   \right.
 	   $
 	     $
 	   r_0^2(1,a,b)=\left\{
 	   \begin{aligned}
 	   	&1&& if \quad (a,b)=(a_{11},b_{11}),\\	   &0&& if \quad (a,b)=(a_{11},b_{12}),\\
 	   	&1&& if \quad (a,b)=(a_{12},b_{11}),\\	   &1&& if \quad (a,b)=(a_{12},b_{12}),\\
 	   \end{aligned}
 	   \right.
 	   $\\
 	    $
 	   r_n^1(1,a,b)=\left\{
 	   \begin{aligned}
 	   	&1&& if \quad (a,b)=(a_{11},b_{11}),\\	   &1&& if \quad (a,b)=(a_{11},b_{12}),\\
 	   	&0&& if \quad (a,b)=(a_{12},b_{11}),\\	   &0&& if \quad (a,b)=(a_{12},b_{12}),\\
 	   \end{aligned}
 	   \right.
 	   $
 	   $
 	   r_n^2(1,a,b)=\left\{
 	   \begin{aligned}
 	   	&0&& if \quad (a,b)=(a_{11},b_{11}),\\	   &1&& if \quad (a,b)=(a_{11},b_{12}),\\
 	   	&0&& if \quad (a,b)=(a_{12},b_{11}),\\	   &0&& if \quad (a,b)=(a_{12},b_{12}),\\
 	   \end{aligned}
 	   \right.
 	   $ \\for all $n\ge 1$ and $\bm{\lambda}_0=\{2,3\}$.  Obviously, Assumption B holds with $\beta=0.4$.
 	   
 	  { Concerning the numerical calculation of the example, we denote the $\epsilon$-equilibrium (claculated by Algorithm A) by $\bm{\pi}_{\epsilon}:= \{\pi^1_{\epsilon},\pi^2_{\epsilon},\ldots,\pi^N_{\epsilon}\}$ with $\pi^k_{\epsilon}=\{\phi_n^k,n\geq 0\}~ \forall k\in I$. 
    Note that the model  considered in this paper is nonstationary. Therefore, the $\epsilon$-equilibrium is time-dependent Markov multipolicy,  which will lead to a significant  computational load.
    However, by Algorithm A, we only need to calculate $\phi_n^k(\cdot\mid i,\bm{\lambda})\in \mathscr{P}(A_n^k(i))$ for  $n\in\{0,1,\cdots,T(\epsilon)\}$. For each $n>T(\epsilon)$ , $\phi_n^k(\cdot\mid i,\bm{\lambda})$ can be  taken any value. It indicates that  the $\epsilon$-equilibrium (obtained by Algorithm A) is a truncated multipolicy, which only requires calculating the first $T(\epsilon)$ time periods instead of traversing all $n\ge0$.
    
 Then from Step-I of Algorithm A we have Table 1, which shows the   the information regarding the period length $T(\epsilon)$ to $\epsilon$  in Table 1.
Moreover, by Theorem 3, we obtain Figure 1, which 
    gives  the dependence of the number of iterations $G(\epsilon)$  required to calculate the $\epsilon$-equilibrium on $\epsilon$ in Algorithm A.}
    	    \begin{table}[H]
    	\newcommand{\tabincell}[2]{\begin{tabular}{@{}#1@{}}#2\end{tabular}}
    	\centering
    	\caption{The period length $T(\epsilon)$ under different errors $\epsilon$}
    	\begin{tabular}{c cccccccccc}
    		\hline\hline\noalign{\smallskip}
    		$\epsilon$ & 0.1& 0.2 & 0.3&0.4 & 0.5&0.6& 0.7&0.8 & 0.9&1.0  \\
    		\hline
    		$T(\epsilon)$ 	& 10& 9 & 8&7 & 7&6& 6&6 & 6&5 
    		\\
    		\noalign{\smallskip}\hline
    	\end{tabular}
    \end{table}

 	   \begin{figure}[H]
 	   	\begin{minipage}[t]
 	   		{1\linewidth} \centering \includegraphics[scale=0.3]{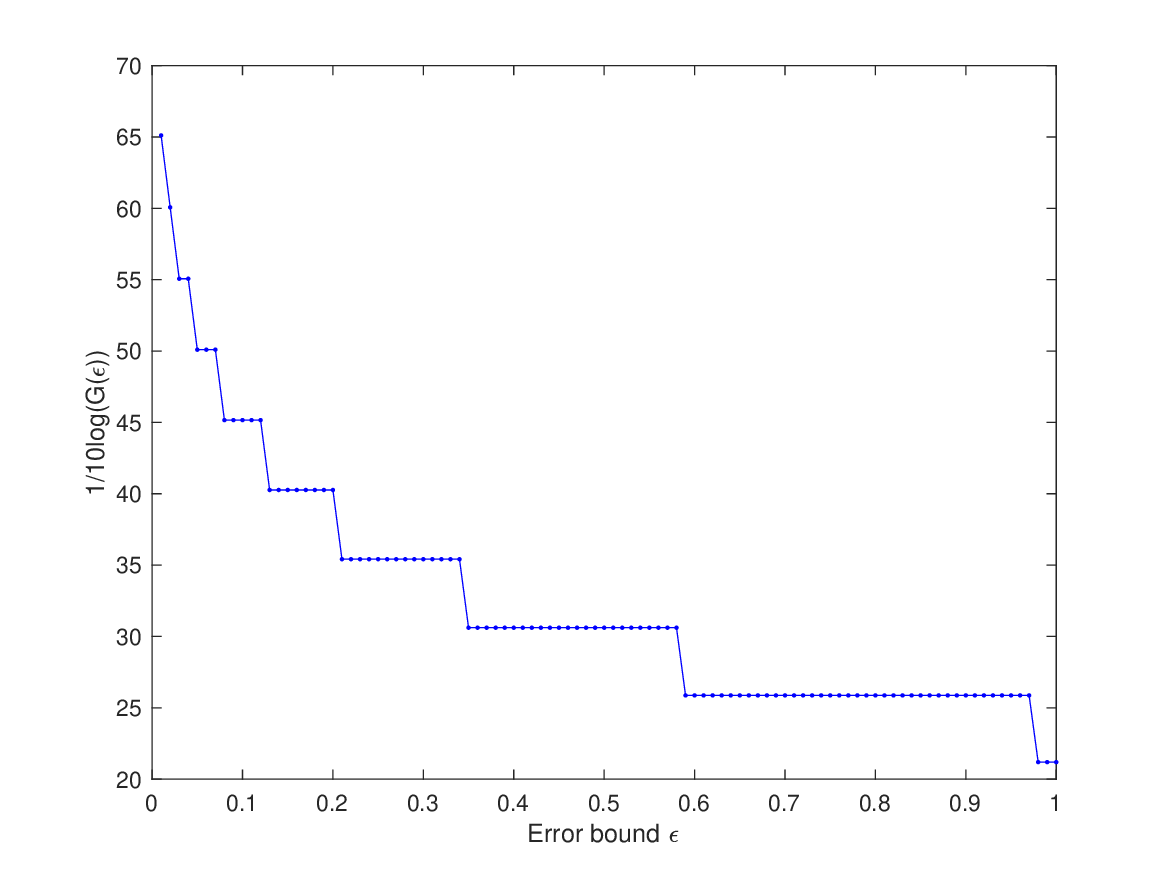}
 	   		\caption{The number of iterations of Algorithm A with respect to $\epsilon$.}
 	   		\label{f1.eps}
 	   	\end{minipage}
 	   \end{figure}
 	  
 	    Furthermore, without loss of generality, taking $\epsilon=0.1$, we can calculate the period length $T(\epsilon)=10$. Moreover,  the existence  of  $0.1$-Nash equilibrium can be obtained by Theorem \ref{thm3} with {in a finite number of iterations}.
 	   Based on the experimental results, for the $0.1$-Nash equilibrium,  we give the several calculated values of $\phi_n^1(a|1,\lambda^1)$ and $\phi_n^2(b|1, \lambda^2)$ in Table 2.
 	\end{example}

 \begin{table}[H]
	\newcommand{\tabincell}[2]{\begin{tabular}{@{}#1@{}}#2\end{tabular}}
	\centering
	\caption{Several values of the 0.1-Nash equilibrium}
	\begin{tabular}{c cccccc}
		\hline\hline\noalign{\smallskip}
		$\bm{\lambda}=(\lambda^1,\lambda^2)$ & $\lambda^1=1$ &$\lambda^1=2$  &~&$\lambda^2=1$ &$\lambda^2=2$&$\lambda^2=3$  \\
		\hline
		$\phi_0^1(a_{11}\mid 1, \bm{\lambda})$ 	& $8.46e-4$&$0.0010$&$\phi_0^2(b_{11}\mid 1, \bm{\lambda})$ &$5.04e-4$&$7.59e-4$&$5.57e-4$
		\\
		\hline
		$\phi_4^1(a_{11}\mid 1, \bm{\lambda})$ & 0.0010&0.0011&$\phi_8^2(b_{11}\mid 1, \bm{\lambda})$ &0.0011&$3.26e-4$&$9.55e-4$
		\\
		\hline
		$\phi_9^1(a_{11}\mid 1, \bm{\lambda})$ & $7.65e-4$&0.0012&$\phi_9^2(b_{11}\mid 1, \bm{\lambda})$ &$2.83e-4$&$7.30e-4$&0.0010
		\\
		\noalign{\smallskip}\hline
	\end{tabular}
\end{table}}
\section*{Declaration of competing interest}

None.

\bibliographystyle{model1-num-names}

\end{document}